\documentclass[11pt]{amsart}%{scrartcl}%{amsart}

\usepackage{euscript}
\usepackage{amsmath}
\usepackage{amsthm}
\usepackage{epsfig}
\usepackage{amssymb}\usepackage{amscd}
\usepackage{epic}
\usepackage{color}
\numberwithin{equation}{section}
\numberwithin{equation}{subsection}

\theoremstyle{plain}

\newtheorem{lemma}[equation]{Lemma}

\newtheorem{thm}[equation]{Theorem}
\newtheorem{cor}[equation]{Corollary}

\newtheorem{prop}[equation]{Proposition}

\theoremstyle{definition}

\newtheorem{defn}[equation]{Definition}

\numberwithin{equation}{section}
\numberwithin{equation}{subsection}

\newcommand{ \lk }{ \mbox{lk} }

\def\C{\mathbb C}

\def\R{\mathbb R}
\def\Z{\mathbb Z}

\newcommand{\calO}{{\mathcal O}}

\newcommand{\labelpar}{\label}

\title{The boundary of the Milnor fibre
of certain non--isolated singularities}

\author{Andr\'as N\'emethi}
\thanks{Both authors  were partially supported by NKFIH Grant  112735 and
ERC Adv. Grant LDTBud of A. Stipsicz at R\'enyi Institute of Math., Budapest.
GP was also supported by `Lend\"ulet' program `LTDBud' at R\'enyi
Institute.}
\address{Alfr\'ed R\'enyi Institute of Mathematics,
Hungarian Academy of Sciences,
Re\'altanoda utca 13-15, H-1053, Budapest, Hungary \newline
 \hspace*{4mm} ELTE - University of Budapest, Dept. of Geometry, Budapest, Hungary \newline \hspace*{4mm}
BCAM - Basque Center for Applied Math.,
Mazarredo, 14 E48009 Bilbao, Basque Country – Spain}
\email{nemethi.andras@renyi.mta.hu }

\author{Gerg\H{o} Pint\'er}
\address{A. R\'enyi Institute of Mathematics, 1053 Budapest,
Re\'altanoda u. 13-15, Hungary.}
\email{pinter.gergo@renyi.mta.hu}


\date{}

\begin{document}

\begin{abstract}
Let $ \Phi: (\C^2, 0) \to ( \C^3, 0) $ be a
finitely determined complex analytic germ and let $(\{f=0\},0)$ be
the reduced equation of its image, a non--isolated hypersurface singularity.
We provide the plumbing graph of the boundary of the Milnor fibre of $f$ from
the double--point--geometry of $\Phi$.
\end{abstract}

\maketitle

%\begin{center}
% {\em Dedicated to }
%\end{center}

\pagestyle{myheadings} \markboth{{\normalsize
A. N\'emethi and G. Pint\'er}}{ {\normalsize The boundary of the Milnor fibre}}

\section{Introduction}

The Milnor fibre of an isolated hypersurface surface singularity is well studied
and it is rather well understood. It has the homotopy type of a bouquet
of 2--spheres, it is an oriented  smooth 4--manifold whose boundary
is diffeomorphic with the link of the singular germ and also with the boundary
of (any) resolution of the germ \cite{MBook}.
This boundary is a  plumbed
3--manifold and one can take as a plumbing graph any of the resolution graphs.
It is the basic bridge between the Milnor fibre and the resolution (both of them being complex analytic fillings of it), and this way  produces several
nice formulas connecting the invariants of these fillings. Here primarily
we think about formulas of Laufer \cite{Laufer77b} or Durfee \cite{Du} and their generalizations, see e.g. \cite{Wa}.

For non--isolated hypersurface singularities in $(\C^3,0)$ the situation is
more complicated. First of all, the link of the germ is not smooth,
hence the boundary of the Milnor fibre cannot be isomorphic with it.
Moreover, a (any) resolution is in fact the resolution of the normalization
(which might contain considerable less information than what one needs in order to
recover the Milnor fibre $F$, or the Milnor fibre boundary $\partial F$),
see e.g. \cite{Si,NSz}. For example (see the present article),
it can happen that the normalization is smooth, while $\partial F$ is rather complicated.
However, the boundary of the Milnor fibre is still a
plumbed 3--manifold, and one expects that its plumbing graph
 codifies considerable information about the germ.
 $\partial F$  can be obtained by surgery of two pieces: one of
them is the boundary of the resolution of the normalization, the other one is related with the transversal singularities associated with the singular curves of the hypersurface singularity \cite{Si,NSz,MP}. In particular, the boundary
of the Milnor fibre plays the same crucial role as in the isolated singularity case (in fact, it is the unique object in this case, which might fulfill this role):
it is the first step in the description of the Milnor fibre, and it is the bridge in the direction of the resolution and the transversal types of the components of the singular locus.

\cite{NSz} presents a
general algorithm, which provides the boundary of the Milnor fibre $\partial F$
for any non--isolated hypersurface singularity in
$(f^{-1}(0),0)\subset (\C^3,0)$.
However, this algorithm uses (some information from)
the embedded resolution of this pair, hence it is rather technical and
in concrete examples is rather computational. Therefore, for particular families
of singularities it is preferable to find more direct description of the plumbing graph
 of $\partial F$ directly from the peculiar intrinsic geometry of the germ.
For  several examples in the literature see e.g.
\cite{NSz} (homogeneous singularities, cylinders of plane curves,
$f=zf'(x,y)$, $f=f'(x^ay^b,z)$), \cite{Baldur} ($f=g(x,y) + zh(x,y)$);
or for other classes consult also
\cite{MP2} and \cite{dBM}.

In this note we assume that $(f^{-1}(0),0)$ is the image
of a finitely determined complex analytic germ
$ \Phi: (\C^2, 0) \to ( \C^3, 0) $.
This means that $  \Phi  $ is a
 generic immersion off the origin, cf. \cite{Wall,Mond2}. Note that for such germs there is a notion similar to the Milnor fibre. This is the disentanglement, the image of a stabilization of $ \Phi $. In some sense it plays the role of the Milnor fibre, e.g. its (singular) boundary is the image of the stable immersion $ \Phi |_{ S^3 }: S^3 \looparrowright S^5 $ associated with $ \Phi $ \cite{NP}, and it is homotop equivalent with a wedge of $2$-spheres \cite{Mondvan, Mondwh}. For such germs $ \Phi $ the disentanglement has a huge literature, it is even more studied than the Milnor fibre.

The main result of the article provides the plumbing graph of $\partial F$,
as a surgery, starting
from the embedded resolution graph  of the double points $(D,0)\subset
(\C^2,0)$. The needed additional pieces, which will be glued
to this primary object
correspond to certain fibre bundles over $S^1$
with fibres the local Milnor fibre of the
 transversal singularity type (and monodromy the corresponding
vertical monodromy). The surgery itself is
 characterized by some homologically determined integers
 combined with  the newly defined {\it `vertical index'}
associated with the irreducible components of the singular locus of $f^{-1}(0)$.

We believe that the present method can serve as a prototype for further more
general families as well.

The structure of the article is the following. In
Section 2 we introduce the notations, in Section 3 we describe by several characterizations the
surgery pieces associated with the components of the singular locus. In Section 4
we describe  the gluing, while the last section contains several concrete
examples (the families are taken from the Mond's list  of simple germs \cite{Mond2}).

Below $ \Z_n\langle g\rangle$ denotes the cyclic group of order $n$ generated by $g$.
We write $\partial _{x_i}f$ instead of $\partial f/\partial x_i$.
%\subsection{}
%
%\begin{enumerate}
% \item $ a \simeq b $: Two smooth manifolds $a $ and $ b $ are (orientation preserving)
% diffeomorphic.
% \item $ a \cong b $: Two algebraic stuctures $ a$ and $ b$ of the same type are isomorphic.
% \item $ G\langle g\rangle$ denoted the cyclic group generated by $g$.
%\end{enumerate}

%\vspace{1mm}

%{\bf Acknowledgement.} The second author wishes to thank Juan  Nu\~no-Ballesteros
%and David Mond for the very informative discussions, and for the warm hospitality of the Universities of
%Valencia and Warwick during his visit.

\section{Preliminaries}

\subsection{} Let $ \Phi: (\C^2, 0) \to ( \C^3, 0) $ be a
 complex analytic germ singular only at the origin.
 We assume that $ \Phi $ is \emph{finitely determined}.
 This happens exactly when  $  \Phi  $ is a {\it
 generic immersion} off  the origin, that is, off the origin it has only single and double values and at
 each double value the intersection of the two smooth branches is
 transversal, cf.
 \cite[Theorem 2.1]{Wall} or \cite[Corollary 1.5]{Mond2}.

Write $(X,0):=({\rm im}(\Phi),0)$ and
let $ f: ( \C^3, 0) \to ( \C, 0) $ be the reduced equation of $(X,0)$.
 Note that $ (X, 0) $ is a {\it non-isolated} hypersurface singularity, except when $ \Phi $ is a regular map (see \cite{NP}).
 We denote by $ (\Sigma,0) = ( \partial_{x_1} f, \partial_{x_2} f, \partial_{x_3} f)^{-1} (0) \subset (\C^3,0) $
 the {\it reduced} singular locus of  $(X,0)$
 (which equals the closure of the set of double values of $\Phi$),
 and by $(D,0)$ the {\it reduced}
 double point curve $ \Phi^{-1} (\Sigma ) \subset (\C^2,0) $.
 (In fact, the finite determinacy of the germ $ \Phi$
 is equivalent with the fact that the double point curve is reduced;  see e.g. \cite{nunodouble}.)

Let $B^6_\epsilon$ be the $\epsilon$--ball in $\C^3$ centered at the origin,
and $ S^5_\epsilon$ its boundary ($\epsilon\in {\mathbb R}_{>0}$). Then for $\epsilon$ sufficiently small $(B_\epsilon^6,0)$ is a Milnor ball for the pair
$(\Sigma,0)\subset (X,0)$, and, furthermore,
$\mathfrak{B}_\epsilon:=\Phi^{-1}(B_\epsilon^6)$
is a (non--metric) $C^\infty$ ball in $(\C^2,0)$, which might serve as a Milnor
ball for $(D,0)$, cf. \cite{looijenga}. We set $ \mathfrak{S}^3 = \Phi^{-1} (S^5_{ \epsilon})=
\partial \mathfrak{B}_\epsilon$, diffeomorphic to $S^3$,
and we treat it as the usual Milnor--ball boundary 3--sphere.
Recall that the  immersion associated with
$ \Phi $ at the level of local neighbourhood boundaries
is $ \Phi |_{ \mathfrak{S}^3} : \mathfrak{S}^3 \to S^5 $ \cite{NP}.

\subsection{Components and links of $\Sigma$ and $D$.}
Let $ \Upsilon \subset S^5_\epsilon$ be the link of $\Sigma$.
It is exactly the  set of double values  of $  \Phi |_{ \mathfrak{S}^3} $.
Let $ L = \Phi^{-1}(\Upsilon) \subset \mathfrak{S}^3 $ denote the set of double points of
$  \Phi |_{ \mathfrak{S}^3} $, that is,
$ L\subset  \mathfrak {S}^3  $ is the link of $D$.
Assume that the reduced equation of $D$ is
$ d: (\C^2, 0) \to (\C, 0) $, whose irreducible decomposition is
$ d= \Pi_{i=1}^l d_i $. The irreducible components of $ D$ are denoted by
$ (D_i, 0) = d_i^{-1} (0) $
and their link components in $L$ by $L_i$,
 $1\leq  i\leq  l$.
$ D $ is equipped with an involution $ \iota: D \to D $ which pairs the double points.
 $ \iota|_{ L} $ induces a permutation (pairing) $ \sigma$ of $ \{1, 2, \dots , l \} $, such that
 $ \iota (L_i) = L_{ \sigma(i)} $. Moreover,
 $ \Phi|_{ L }: L \to \Upsilon $ is a double covering with
 $ \Phi (L_i)=\Phi (L_{ \sigma(i) }) $.
 If $ i = \sigma(i) $ for some $i$, then $ \Phi |_{L_i} $ is a nontrivial double covering
  of its image, while  above  the other components the covering  is  trivial.
 Let $ J $ be the set of pairs   $\{i, \sigma (i)\}$ ($1\leq i \leq l$), it
  is the index set of the   components of $ \Upsilon $; they will
  be denoted by $ \{\Upsilon_j\}_{ j \in J}$.

All link components are considered with  their natural orientations.

\subsection{The embedded resolution of $ (D,0)\subset (\C^2,0) $}\label{ss:embres}
Let $\pi: (\widetilde{\C^2}, E) \to ( \C^2, 0) $ be a \emph{good embedded resolution} of $ ( D , 0)$,
$ E=\pi^{-1}(0)$ the reduced exceptional  curve,
$\bigcup_{v \in V} E_v $ the irreducible components of $E$, each diffeomorphic to  $ \C {\mathbb P}^1 $
intersecting each other transversally.
Let $ e_v \in \Z $ denote the self--intersection number
of $ E_v $ in $ \widetilde{ \C^2}$. Clearly $e_v < 0 $ for all $ v$.

We denote the \emph{strict transform}
$\overline{(d \circ \pi)^{-1} (0) \setminus E} \subset \widetilde{ \C^2}$
 of $ D $ by
$\tilde{D}$, and its components by  $ \tilde{D}_i  $,
 $1\leq i\leq l$. Each $ \tilde{D}_i  $ intersects (transversally)
 only one exceptional component, say $ E_{v(i)} $.

The \emph{total transform} of $ D_i $ is the divisor
 $ {\rm div}(d_i \circ \pi)= \Sigma_{v \in V} m_i(v)
 \cdot E_v + \tilde{D}_i $,
where the multiplicity $ m_i(v) \in \Z_{>0} $ is the
vanishing order of  $d_i \circ \pi $ along $ E_v $.

 Let $ \Gamma $ be the embedded resolution graph of $ (D, 0) \subset ( \C^2, 0) $ associated with the resolution ${\pi}$. The strict transforms
$ \{\tilde{D}_i\}_{i=1}^l  $ will be codified (as usual) by
arrowhead vertexes $\{a_i\}_{i=1}^l$.

The multiplicities $m_i(v)$ ($v \in V $, $1\leq i\leq l$) are determined by $ \Gamma $ via the identities (see e.g. \cite{Lauferbook,Five})
\begin{equation}\labelpar{eq:mult}
 \Sigma_{v \in V} m_i(v) (E_v \cdot E_w)  + (\tilde{D}_i\cdot E_w) = 0 \ \ \mbox{for all $w\in V$.}
 \end{equation}

\section{The manifold $ Y $}\labelpar{s:Y}
In the construction of the boundary of the Milnor fibre we will need a special
3--manifold with torus boundary. Its several realizations and properties
will be discussed in this section.

In the sequel $S^1$ (as the boundary of the unit disc of $\C$) and the real interval
$I:=[-1,1]$ are considered with their natural orientations;
 $\bar{\cdot }$ denotes  the complex conjugation of $\C$.

\subsection{The definition  of $Y$}\labelpar{ss:cY} Consider the $\Z_2$--action on $ S^1 \times S^1 \times I$ defined by the involution
%\begin{equation}\labelpar{eq:act}
 $(x, y, z) \mapsto (-x, \bar{y}, -z)$, % \mbox{,}
%\end{equation}
and define $Y$ as the factor
\begin{equation}\labelpar{eq:Y}
 Y= \frac{S^1 \times S^1 \times I}{(x,y,z) \sim (-x, \bar{y}, -z)}.
\end{equation}

$Y$ is a $3$--manifold with a boundary diffeomorphic  with $S^1\times S^1$.
The  projections to different  components provide
 different `realizations' of $Y$.

\vspace{2mm}

(1) The projection  to the first coordinate $x$ gives a fibration
\begin{equation}\labelpar{eq:cyl}
 \begin{array}{ccc}
  S^1 \times I & \rightarrow & Y \\
             &              & \downarrow \\
             &              & S^1
 \end{array}
\end{equation}
where the base space $ S^1=S^1/\{x \sim -x\}  $
is parametrized by $ x^2 $,  and the monodromy
diffeomorphism $ S^1 \times I \to S^1 \times I $
over the base space is  $ (y, z)  \mapsto ( \bar{y} , -z) $.

\vspace{2mm}

(2) The projection  to the first two coordinates $(x,y)$ realizes $Y$ as the total space of a fibration
\begin{equation}\labelpar{eq:klein}
 \begin{array}{ccc}
   I & \rightarrow & Y \\
             &              & \downarrow \\
             &              & \mathcal{K}
 \end{array}
\end{equation}
with  fibre $I$ and  base space
%\begin{equation}
$ \mathcal{K} := (S^1 \times S^1)/\{(x,y) \sim (-x, \bar{y})\}$,
%\end{equation}
the Klein bottle.
The  factorization $S^1\times S^1\to \mathcal{K}$
 is the orientation double cover of $\mathcal{K} $. In particular,
the fibration (\ref{eq:klein})
is the segment bundle of the orientation line bundle of $ \mathcal{K}$,
hence  the orientation double cover of $ \mathcal{K} $ is realized also
 by the restriction of the bundle map to the boundary
$ \partial Y \approx S^1 \times S^1 $.

\vspace{2mm}

(3) The projection to the $(x, z)$ coordinates realizes $Y$ as the total space of a fibration
\begin{equation}\labelpar{eq:mob}
 \begin{array}{ccc}
  S^1  & \rightarrow & Y \\
             &              & \downarrow \\
             &              & \mathcal{M}
 \end{array}
\end{equation}
over the base space
%\begin{equation}
$ \mathcal{M} := (S^1 \times I)/\{(x,z) \sim (-x, -z)\}$,
%\end{equation}
the M\"obius band.
In this way, $Y$ appears as the
tangent circle bundle of $\mathcal{M}$, i.e. as
the sub--bundle of the tangent bundle $ T \mathcal{M} $ consisting of unit tangent vectors.
This follows from the fact that both circle bundles have  the same monodromy map
 along the midline of $ \mathcal{M} $, namely $S^1 \to S^1$, $y \mapsto \bar{y}$.

\vspace{2mm}

(4) The projection  to the $(y,z)$ coordinates
realizes $Y$ as the total space of a projection
\begin{equation}\labelpar{eq:seif}
 \begin{array}{ccc}
   S^1 & \rightarrow & Y \\
             &              & \downarrow \\
             &              & D^2
 \end{array}
 \end{equation}
 to the  base space
%\begin{equation}
 $(S^1 \times I)/\{(y,z) \sim (\bar{y}, -z)\}$,
%\end{equation}
 the 2--disc $ D^2 $. Although the involution $ (y,z) \mapsto (\bar{y}, -z) $ has two fix points $(-1,0) $ and $(1,0)$,  the factor $D^2$ can be smoothed.
 However, the projection $Y\to D^2$  is not a locally trivial fibration:
 it is a Seifert fibration with two exceptional fibres sitting
  above $(-1,0) $ and $(1,0)$.   We will refer to this $S^1$--fibration
  as the \emph{canonical Seifert fibration} of $ Y$.

 The Seifert invariants of the exceptional fibres can be calculated as in
  p. 307 of \cite{neumann1}.
The two exceptional fibres of the canonical Seifert fibration \eqref{eq:seif}
can be seen in the projection (3) as well: they correspond to the
tangent vectors of the midline. On the other hand, a generic orbit consists of those
unit tangent vectors of $\mathcal{M}$, which form non--zero
angle $\pm \alpha$ (with $\alpha$ fixed) with the  midline.
Thus,  a generic fibre in a neighbourhood of an exceptional fibre
goes around twice and  both  Seifert invariants are $(2,1)$.
Hence, cf. \cite{neumann1}, a plumbing graph of $Y$ is:

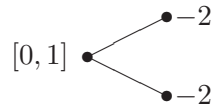
\begin{picture}(300,50)(-100,35)
\put(100,60){\circle*{4}}
\put(130,75){\circle*{4}}
\put(130,45){\circle*{4}}
\put(100,60){\line(2,1){30}}
\put(100,60){\line(2,-1){30}}
\put(82,60){\makebox(0,0){$[0,1]$}}
\put(140,75){\makebox(0,0){$-2$}}
\put(140,45){\makebox(0,0){$-2$}}
\end{picture}

\noindent
Here $[0,1]$ denotes a genus $0$ core-space with one disc removed.
The Euler number of
the $S^1$-bundle corresponding to the middle vertex is irrelevant, the resulted
$3$-manifolds with boundary are diffeomorphic with each other
 (hence with $Y$ too).
However,  the restriction of the canonical Seifert fibration to $ \partial Y
\simeq S^1 \times S^1 $ determines an $S^1$--fibration of the boundary.
Moreover, $Y$ admits a  unique  closed Seifert $3$--manifold $ \bar{Y}$,
from which $Y$ can be
 obtained by omitting a tubular neighborhood of a generic fibre (that is,
 $ \bar{Y}$ is obtained by extending the fibration of $\partial Y$ to an
 $S^1$--fibration {\it without any  new special Seifert fibres}.
In this way there is a
canonical choice for the Euler number of the `middle' vertex, which is the
`middle' Euler number of the plumbing graph of $ \bar{Y}$.
We will calculate it below.
Using this  Euler number, the graph also determines
a parametrization (framing) of $ \partial Y\simeq S^1\times S^1$.

\subsection{Homotopical properties of $ Y $}\label{ss:homY}
In order to understand better the structure of $Y$
 we consider its fundamental domain  (in coordinates $(s,t,z)$):

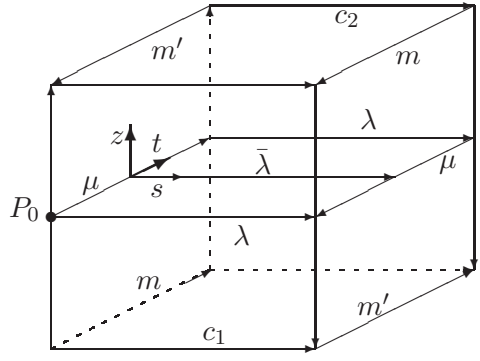
\begin{picture}(300,160)(0,-5)
\put(100,10){\vector(1,0){100}}
\put(100,10){\vector(0,1){100}}
\put(100,110){\vector(1,0){100}}
\put(200,110){\vector(0,-1){100}}
\put(100,60){\vector(1,0){100}}
\put(160,140){\vector(-2,-1){60}}
\put(260,140){\vector(-2,-1){60}}
\put(160,140){\vector(1,0){100}}
\put(200,10){\vector(2,1){60}}
\put(260,140){\vector(0,-1){100}}
\dashline{2}(160,40)(160,140)
\dashline{2}(160,40)(260,40)
\dashline{2}(100,10)(160,40)
\put(100,60){\vector(2,1){60}}
\put(260,90){\vector(-2,-1){60}}
\put(160,90){\vector(1,0){100}}
\dashline{4}(130,75)(230,75)
\put(130,75){\vector(1,0){100}}

\put(151,35){\vector(2,1){10}}
\put(160,136){\vector(0,1){4}}
\put(256,40){\vector(1,0){4}}

\put(125,90){\makebox(0,0){$ z $}}
\put(140,70){\makebox(0,0){$ s $}}
\put(140,88){\makebox(0,0){$ t $}}
\put(130,75){\vector(1,0){20}}\thicklines
\put(130,75){\vector(2,1){15}}\thicklines
\put(130,75){\vector(0,1){20}}\thicklines
\put(115,72){\makebox(0,0){$ \mu $}}
\put(250,80){\makebox(0,0){$ \mu $}}
\put(172,53){\makebox(0,0){$ \lambda $}}
\put(220,97){\makebox(0,0){$ \lambda $}}
\put(180,81){\makebox(0,0){$ \bar{\lambda} $}}

\put(222,27){\makebox(0,0){$ m' $}}
\put(137,35){\makebox(0,0){$ m $}}
\put(235,120){\makebox(0,0){$ m $}}
\put(144,124){\makebox(0,0){$ m' $}}

\put(162,15){\makebox(0,0){$ c_1 $}}
\put(212,135){\makebox(0,0){$ c_2 $}}
\put(100,60){\circle*{4}}
\put(90,63){\makebox(0,0){$ P_0 $}}
\end{picture}

\noindent
where $ s \in [0, \pi] $, $ t \in [- \pi, \pi ] $ and $ z \in I= [-1, 1] $. The original coordinates  are $x= e^{si} $ and $ y= e^{ti} $. Then $Y$ is obtained  by the following  identification of the sides of the cube:
\begin{equation*}
 (0, t, z) \sim ( \pi, -t, -z) \mbox{ and } (s, -\pi, z) \sim (s, \pi, z) \mbox{.}
\end{equation*}
The boundary is
\begin{equation*}
 \partial Y = \frac{[0, \pi] \times [- \pi, \pi ] \times \{-1, 1 \}}{(0,t,\pm 1)
 \sim (\pi, -t, \mp 1), \ (s, -\pi, \pm 1 ) \sim (s, \pi, \pm 1)} \simeq S^1 \times S^1 \mbox{.}
\end{equation*}
%
%\vspace{2mm}
%
The $ S^1 $-action determining the canonical Seifert fibration (\ref{eq:seif}) is
induced by the translation along the $s$--axis. The exceptional fibres are
\begin{equation*}
 \lambda = \{(s, -\pi, 0) \ | \ s \in [0, \pi] \} = \{(s, \pi, 0) \ | \ s \in [0, \pi] \} \ \mbox{ and} \
%\end{equation*}
%\begin{equation*}
 \bar{\lambda} = \{(s, 0, 0) \ | \ s \in [0, \pi] \}  \mbox{.}
\end{equation*}
Any fixed $(t,z) \notin \{ (0, 0), (\pm \pi, 0) \} $ determines a generic fibre in the form
\begin{equation*}
 \{(s, t, z) \ | \ s \in [0, \pi] \} \cup \{ (s, -t, -z) \ | \ s \in [0, \pi ] \} \mbox{,}
\end{equation*}
where $ (0, \pm t, \pm z) $ are glued together with $(\pi, \mp t,  \mp z) $.
For example, $c= c_1 \cup c_2 \subset \partial Y$ is a generic fibre.

The base space $D^2$ of the canonical Seifert-fibration can be represented as

\begin{equation*}
 \frac{\{0 \} \times [- \pi, \pi ] \times [-1, 0]}{(0, - \pi, z)
 \sim (0, \pi, z) \mbox{ , } (0, t, 0) \sim (0, -t, 0)} \mbox{.}
\end{equation*}
Its boundary is the class of $m$, a circle.

Next, we describe the fundamental group and the homology of $ Y $. % and $ \bar{Y} $.
$ Y $ is homotopically equivalent with the Klein bottle $ \mathcal{K} $.
Let us choose the base point $ P_0=(0,-\pi, 0) $. Both four vertexes of
the rectangle representing $ \mathcal{K} $ represents $ P_0$.
Thus the fundamental group of $ Y $ can be presented as
 \begin{equation}\labelpar{eq:fundY}
  \pi_1 (Y) = \langle \mu, \lambda \ | 	\  \mu \cdot \lambda \cdot \mu   = \lambda \rangle \mbox{,}
  \end{equation}
 where $ \mu $ and $ \lambda $ denote also the class of $ \mu $ and $ \lambda $ in $ \pi_1 (Y) = \pi_1 ( \mathcal{K})$;
 cf. with the description (2) from \ref{ss:cY}). A more precise description can be given via the next diagrams,
 provided by the $\{z=0\}$ subspace of $Y$ (which can be identified with ${\mathcal K}$).

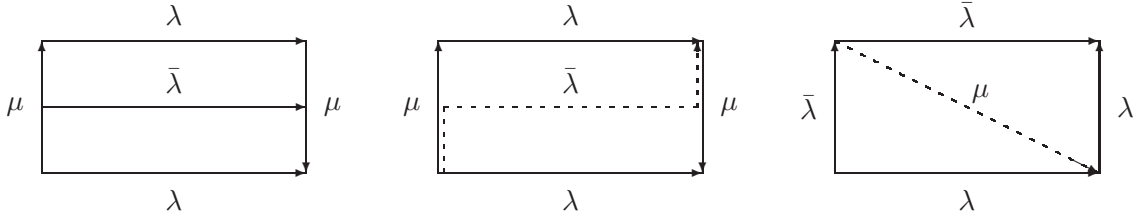
\begin{picture}(300,90)(0,-5)
\put(0,60){\vector(1,0){100}}
\put(0,10){\vector(0,1){50}}
\put(0,10){\vector(1,0){100}}
\put(100,60){\vector(0,-1){50}}
\put(50,70){\makebox(0,0){$ \lambda $}}
\put(50,0){\makebox(0,0){$ \lambda $}}
\put(-10,35){\makebox(0,0){$ \mu $}}
\put(110,35){\makebox(0,0){$ \mu $}}
\put(0,35){\vector(1,0){100}}
\put(50,45){\makebox(0,0){$ \bar{ \lambda} $}}

\put(150,60){\vector(1,0){100}}
\put(150,10){\vector(0,1){50}}
\put(150,10){\vector(1,0){100}}
\put(250,60){\vector(0,-1){50}}
\put(200,70){\makebox(0,0){$ \lambda $}}
\put(200,0){\makebox(0,0){$ \lambda $}}
\put(140,35){\makebox(0,0){$ \mu $}}
\put(260,35){\makebox(0,0){$ \mu $}}
\dashline{2}(152,35)(248,35)
\dashline{2}(248,35)(248,60)
\dashline{2}(152,10)(152,35)
\put(248,56){\vector(0,1){4}}
\put(200,45){\makebox(0,0){$ \bar{ \lambda} $}}

\put(300,60){\vector(1,0){100}}
\put(300,10){\vector(0,1){50}}
\put(300,10){\vector(1,0){100}}
\put(400,10){\vector(0,1){50}}
\put(350,70){\makebox(0,0){$ \bar{\lambda} $}}
\put(350,0){\makebox(0,0){$ \lambda $}}
\put(290,35){\makebox(0,0){$ \bar{\lambda} $}}
\put(410,35){\makebox(0,0){$ \lambda $}}
\dashline{2}(300,60)(400,10)
\put(390,15){\vector(2,-1){10}}
\put(355,40){\makebox(0,0){$ \mu $}}
\end{picture}

\noindent
The first diagram shows homological cycles.
 In order to rewrite the fundamental group,
 let $ \bar{ \lambda} $ be the closed path  shown in the second diagram
  by  the dashed line. Then $ \bar{\lambda} = \mu \cdot  \lambda $ in $ \pi_1 ( Y) $. Note that
 $ \lambda^2 = \mu \cdot \lambda \cdot \mu \cdot \lambda =  \lambda \cdot \mu \cdot \lambda \cdot \mu $,
 thus $ \lambda^2 = \bar{\lambda}^2 = \mu^{-1} \cdot \lambda^2 \cdot \mu $
 and this element  commutes with $ \mu $.
The fundamental group can be also presented as
 \begin{equation}\labelpar{eq:fundY2}
  \pi_1 (Y) = \langle \lambda, \bar{\lambda} \ | 	\  \bar{\lambda}^2    = \lambda^2 \rangle \mbox{,}
 \end{equation}
according to the third  picture above.

%\begin{picture}(300,90)(150,-5)
%\put(300,60){\vector(1,0){100}}
%\put(300,10){\vector(0,1){50}}
%\put(300,10){\vector(1,0){100}}
%\put(400,10){\vector(0,1){50}}
%\put(350,70){\makebox(0,0){$ \bar{\lambda} $}}
%\put(350,0){\makebox(0,0){$ \lambda $}}
%\put(290,35){\makebox(0,0){$ \bar{\lambda} $}}
%\put(410,35){\makebox(0,0){$ \lambda $}}
%\dashline{2}(300,60)(400,10)
%\put(396,12){\vector(2,-1){4}}
%\put(355,40){\makebox(0,0){$ \mu $}}
%\end{picture}

 On the other hand, the fundamental group of the boundary is
\begin{equation}\labelpar{eq:pY}  \pi_1 ( \partial Y ) = H_1 (\partial Y , \Z ) \cong \Z \langle m \rangle
\oplus \Z \langle c \rangle \mbox{.} \end{equation}
The $ \partial Y \simeq S^1\times S^1 \hookrightarrow Y $ embedding (which is homotically the same as the orientation covering
$S^1\times S^1\to \mathcal{K} $) induces a monomorphism $ \pi_1 (\partial Y) \to \pi_1 (Y) $. It is determined by the images of the generators, which are
\begin{equation}\label{eq:cmu}
 m \mapsto \mu \mbox{ and } c \mapsto \lambda^2= \bar{ \lambda}^2 \mbox{.}
\end{equation}
A direct computation shows that $ [ \pi_1 (Y) , \pi_1(Y) ] = \Z \langle \mu^2 \rangle $ and
\[
 H_1 (Y, \Z ) \cong \Z \langle \lambda \rangle \oplus \Z_2 \langle \mu \rangle \mbox{.}
\]
Note that $m=m'$ in $H_1(\partial Y,\Z)$, and analysing  $\{s=0\}\subset Y$ one obtains that  $m=-m'$ in $H_1(Y,\Z)$.
Hence the class of $m$ in $H_1(Y,\Z)$ has order 2, it is exactly $\mu$.

The next Lemma shows that the classes $m$ and $c$ in $H_1(\partial Y,\Z)$ have certain universal
properties with respect to the inclusion $\partial Y\subset Y$.

\begin{lemma}\label{lem:UNIV}
(a) $\pm m$ are the unique primitive elements of $H_1(\partial Y,\Z)$
with the property that their doubles vanish in $H_1(Y,\Z)$.

(b) $\pm c$ are the unique primitive elements of $H_1(\partial Y,\Z)=\pi_1(\partial Y)$
whose images in $\pi_1(Y)$ are in the center of $\pi_1(Y)$.
\end{lemma}
\begin{proof} (a) is clear. For (b) first note
 that any element of $\pi_1(Y)$ can be written in the from $\lambda^k\mu^l$ for some $k,l\in\Z$, and then using this one
 verifies that
 the center of $\pi_1(Y)$ is $\langle \lambda^2\rangle$.
\end{proof}
\subsection{Homological properties of  $\bar{Y}$}\label{ss:bary} The closed Seifert
3--manifold  $\bar{Y}$ considered in \ref{ss:cY}(4) is constructed  as follows.
First we consider a new disc $D_{new}^2$ and the trivial fibration $D_{new}^2\times S^1$.
Then we past  $m$ with the boundary of $D_{new}^2$
 and we extend the canonical Seifert--fibration of $Y$ above this  disc (as base space)
of the  trivial fibration $D_{new}^2\times S^1$. This leads to the Seifert fibred closed manifold
\begin{equation}\labelpar{eq:barY}
 \bar{Y} = \frac{Y \cup
  (D_{new}^2 \times S^1)}{\partial Y\sim \partial D_{new}^2 \times S^1,\  m\sim \partial D_{new}^2 \times *, \
  c\sim *\times  S^1 } \mbox{.}
\end{equation}
$ H_1(\bar{Y},\Z) $ can be determined by  the Mayer-Vietoris sequence of the decomposition (\ref{eq:barY}):
\[
 \begin{array}{ccccccc}
  H_1( \partial Y, \Z) & \to & H_1 (Y, Z) \oplus H_1 (D_{new}^2 \times S^1, \Z) & \to & H_1 ( \bar{Y}, \Z) & \to & 0 \\
  \|            &     & \parallel                                   &   &  \parallel         &      &  \\
  \Z \langle m \rangle \oplus \Z \langle c \rangle & \to & \Z \langle \lambda \rangle \oplus \Z_2 \langle \mu \rangle \oplus \Z \langle c' \rangle & \to & H_1 ( \bar{Y}, \Z) & \to & 0 \\
 %( m,c) & \mapsto & ( \mu, 2 \lambda + c') & & & & \\
 % c & \mapsto &  2 \lambda + c'  & & & & \\
 \end{array}
\]
where $ m \mapsto \mu$ and $c\mapsto  2 \lambda + c'$.
Thus \begin{equation}\label{eq:Z}
 H_1 ( \bar{Y}, \Z) \cong \Z \langle \lambda \rangle .\end{equation}

\subsection{A plumbing graphs of $Y$ and $\bar{Y}$}\label{ss:plY}
By \cite{neumann2},  $ \bar{Y} $ has a plumbing graph $ G $ of the form

 
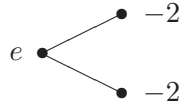
\begin{picture}(300,50)(-100,35)
 \put(100,60){\circle*{4}}
 \put(130,75){\circle*{4}}
 \put(130,45){\circle*{4}}
 \put(100,60){\line(2,1){30}}
 \put(100,60){\line(2,-1){30}}
 \put(90,60){\makebox(0,0){$ e $}}
 \put(145,75){\makebox(0,0){$-2$}}
 \put(145,45){\makebox(0,0){$-2$}}
 \end{picture}

\noindent
(see also the discussion from \ref{ss:cY}(4))
and the Euler number $e$ should be chosen such that $ H_1 (M^3(G), \Z ) \cong \Z$
(cf. (\ref{eq:Z})).
%where $M^k(G)$ denotes the plumbed $k$--manifold associated with  $G$,$k=3$ or 4.
Here $ M^3(G)  $ denotes the plumbed $3$--manifold associated with the graph $G$, it is the boundary of $ M^4 (G) $, the plumbed $4$--manifold associated with $G$. %Consider the Mayer-Vietoris sequence of the pair $ (M^4 (G), M^3(G))$
%\[
% \begin{array}{ccccccc}
%   H_2 (M^4(G), \Z) & \rightarrow & H_2(M^4(G), M^3 (G), \Z) & \to & H_1 (M^3(G), \Z) & \to & 0 \\
% \end{array}
%\]
%where the first map is given by the intersection matrix. Hence $ H_1 (M^3(G), \Z) \cong \Z $ if and only if
%\[
% \det \left(
% \begin{array}{ccc}
% -2 & 0 & 1 \\
% 0 & -2 & 1 \\
% 1 & 1 & x \\
% \end{array}
% \right)
% =0 \mbox{,}
%\]
%thus $ e= -1 $.
This (via the long cohomological exact sequence of the pair
$(M^4(G),M^3(G))$) imposes the degeneracy of the intersection matrix of the plumbing (cf. e.g.
 \cite[15.1.3]{NSz}).  Hence $e=-1$.
 In particular, $\bar{Y}$ (without its Seifert fibration structure) is diffeomorphic to
 $S^1\times S^2$ (which also shows that $\bar{Y}$ admits an orientation reversing diffeomorphism).

 Furthermore,  consider the graph

 
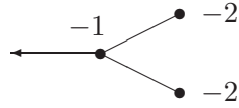
\begin{picture}(300,50)(-100,35)
 \put(100,60){\circle*{4}}
 \put(100,60){\vector(-1,0){34}}
 \put(130,75){\circle*{4}}
 \put(130,45){\circle*{4}}
 \put(100,60){\line(2,1){30}}
 \put(100,60){\line(2,-1){30}}
 \put(95,70){\makebox(0,0){$-1$}}
 \put(145,75){\makebox(0,0){$-2$}}
 \put(145,45){\makebox(0,0){$-2$}}
 \end{picture}

\noindent
The arrow denotes a knot $ K \simeq S^1 \subset  \bar{Y} $, which is a generic
 $S^1$--fibre associated with the middle vertex by the plumbing construction of $\bar{Y}$. %, and it is also a fibre of the canonical Seifert fibration.
 Let $ N(K)^\circ  $ be an open  tubular neighborhood of $ K $ in $ \bar{Y} $.
 Then
 %We proved the following:
%
%\begin{prop} (a)
$ Y \simeq \bar{Y} \setminus N(K)^\circ  $, and
%(b) T
the induced (singular/Seifert) $S^1$--fibration associated with the
 middle vertex (by the plumbing construction) on $Y$ agrees with the canonical Seifert fibration of $Y$.
%\end{prop}
The Euler number $-1$ of the middle vertex %is irrelevant for point (a). It just
determines a parametrization (framing)
of $ \partial Y \simeq S^1 \times S^1 $.

%\subsection{}
We can present $ \pi_1 (Y) $ also from the plumbing graph using the description of \cite{mumford}. Let $ \lambda $, $ \bar{ \lambda} $ and $c$ be
oriented $S^1$--fibres associated with the three vertices provided by the plumbing construction (and extended by convenient
connecting paths to  a base point as in \cite{mumford}).
Next, let $m$ be the meridian of $K$  corresponding to the arrowhead
(and extended by a convenient path to the base point).

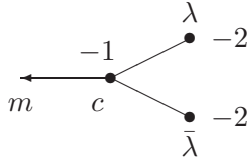
\begin{picture}(300,65)(-100,26)
\put(100,60){\circle*{4}}
\put(100,60){\vector(-1,0){34}}
\put(130,75){\circle*{4}}
\put(130,45){\circle*{4}}
\put(100,60){\line(2,1){30}}
\put(100,60){\line(2,-1){30}}
\put(95,70){\makebox(0,0){$-1$}}
\put(145,75){\makebox(0,0){$-2$}}
\put(145,45){\makebox(0,0){$-2$}}
\put(95,50){\makebox(0,0){$c$}}
\put(130,85){\makebox(0,0){$ \lambda $}}
\put(130,35){\makebox(0,0){$ \bar{ \lambda} $}}
\put(66,50){\makebox(0,0){$m$}}
\end{picture}

\noindent
% We also use these symbols to denote its class in $ \pi_1 (\bar{ Y}) $ and in
%$ \pi_1 (Y) $ (extended with a path from and to the base point).
Then, by \cite{mumford},  there is a choice of the connecting pathss such that
$ \pi_1 (Y) $ is generated by $\lambda, \ \bar{\lambda}, \ c$ and $m$, and they satisfy the relations $\lambda^2=\bar{\lambda}^2=c$, and $c=m\lambda\bar{\lambda}$.
This is compatible with the description from subsection \ref{ss:homY},
cf. (\ref{eq:fundY}) and (\ref{eq:cmu}).

Note that by plumbing calculus (cf. \cite{neumann1})
one has the equivalence of plumbed manifolds (where $\widetilde{\Gamma}$
is any graph):

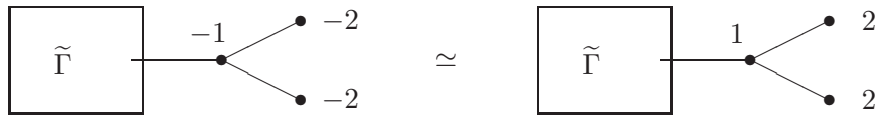
\begin{picture}(300,60)(-30,30)
\put(20,40){\framebox(50,40)}
\put(40,60){\makebox(0,0){$\widetilde{\Gamma}$}}
 \put(100,60){\circle*{4}}
 \put(100,60){\line(-1,0){34}}
 \put(130,75){\circle*{4}}
 \put(130,45){\circle*{4}}
 \put(100,60){\line(2,1){30}}
 \put(100,60){\line(2,-1){30}}
 \put(95,70){\makebox(0,0){$-1$}}
 \put(145,75){\makebox(0,0){$-2$}}
 \put(145,45){\makebox(0,0){$-2$}}

 \put(220,40){\framebox(50,40)}
\put(240,60){\makebox(0,0){$\widetilde{\Gamma}$}}
 \put(300,60){\circle*{4}}
 \put(300,60){\line(-1,0){34}}
 \put(330,75){\circle*{4}}
 \put(330,45){\circle*{4}}
 \put(300,60){\line(2,1){30}}
 \put(300,60){\line(2,-1){30}}
 \put(295,70){\makebox(0,0){$1$}}
 \put(345,75){\makebox(0,0){$2$}}
 \put(345,45){\makebox(0,0){$2$}}

 \put(185,60){\makebox(0,0){$\simeq$}}
 \end{picture}

Hence, $Y$ or $-Y$ spliced along $K$ to any 3--manifold
eventuate  diffeomorphic manifolds.

  % \[
 % \pi_1 (\bar{Y}) = \langle \lambda, \bar{\lambda}, c \ | \ \lambda^2 = \bar{\lambda}^2 = \lambda \cdot \bar{\lambda} = c \rangle \mbox{,}
% \]
% from which $ \lambda= \bar{ \lambda}$ follows, and $ \pi_1 (M^3(G)) = \Z \langle \lambda \rangle = \Z \langle \bar{ \lambda} \rangle $.
%
% Similarly,
% \[
% \pi_1 (Y) = \langle \lambda, \bar{\lambda}, c, m \ | \ \lambda^2 = \bar{\lambda}^2 = \lambda \cdot \bar{\lambda} \cdot m = c \rangle \mbox{,}
% \]
% thus $ \lambda = \bar{ \lambda} \cdot m $ and
% \[ \pi_1 ( Y) = \langle \lambda, \bar{\lambda} \ | \lambda^2 = \bar{\lambda}^2 \rangle \mbox{ .}
% \]

 \section{The boundary of the Milnor fibre}

\subsection{} Let $ F = f^{-1} ( \delta ) \cap B^6_{ \epsilon} $ be the
Milnor fibre of $ f $, where $ \delta \in \C^* $, $| \delta | \ll \epsilon $.
We wish to construct the $3$--manifold $ \partial F = f^{-1} ( \delta )
\cap S^5_{ \epsilon} $ as a surgery of $ S^3 $ along the link $ L $.

Let $ N_i $ be a sufficiently small
tubular neighborhood of $ L_i $ in $ S^3 $.
%The choice of $ N_i $ will be specified later.
For each $j=\{i, \sigma(i) \} $ we define $X_j $ as
\begin{equation}\labelpar{eq:xj}
 X_j =   \left\{ \begin{array}{ccc} S^1\times S^1 \times I & \mbox{ if} &  i \neq \sigma(i) \\
                  Y & \mbox{ if} &  i = \sigma(i), \\
                 \end{array} \right.
                 \end{equation}
where $ Y $ is the $3$--manifold described in Section \ref{s:Y}. Recall
 that $ \partial Y \simeq S^1\times S^1$.

\begin{prop}\label{pr:gl} One has an orientation preserving diffeomorphism
 \begin{equation}\labelpar{eq:mb}  \partial F \simeq \left(\mathfrak{S}^3 \setminus
 \bigcup_{i=1}^l {\rm int}(N_i)  \right)\cup_{ \phi} \left(\bigcup_{j \in J} X_j \right) \mbox{,} \end{equation}
 where $ \phi : \partial (\mathfrak{S}^3\setminus \cup_i\,{\rm int}(N_i))\to
  -\partial (\cup_{j\in J}X_j)$ is a collection $( \phi_j)_{j \in J}$ of diffeomorphisms
% $ \phi_{ \{ i, \sigma(i) \} }: -\partial N_i \cup -\partial N_{\sigma(i)} \to -\partial X_{ \{ i, \sigma(i) \}}$.
 % That is
 \[ \phi_{\{i, \sigma(i)\} }: \left\{ \begin{array}{ccc} -\partial N_i \cup -\partial N_{\sigma(i)}
  \to -\partial
 (S^1\times S^1 \times I)  & \mbox{ if } & i \neq \sigma(i) \\
- \partial N_i \to -\partial Y\hspace{1cm} & \mbox{ if } & i = \sigma(i). \\
                                \end{array} \right.    \]
\end{prop}

\begin{proof} The decomposition follows from the general decomposition proved in \cite{Si}, see also
 \cite[2.3]{NSz}. For the convenience of the reader we sketch the construction.
 Recall that $\Upsilon_j\subset S^5_\epsilon$ is the link of the component $\Sigma_j$ of $\Sigma$,
 $j=\{i,\sigma(i)\}\in J$.
 Consider a sufficiently small tubular  neighborhood $N(\Upsilon_j)$ of it in $S^5_\epsilon$.
 We can assume that $\Phi^{-1}(N(\Upsilon_j))=N_i\cup N_{\sigma(i)}$. Furthermore, for $\epsilon$ small,
 the intersection of $\partial N(\Upsilon_j)$ with $K:=X\cap S^5_\epsilon$ is transversal.
 Therefore, for $0<|\delta|\ll \epsilon$, the intersection of $\partial N(\Upsilon_j)$ with $\partial F$ is
 still transversal in $S^5_\epsilon$, and, in fact, $K\setminus \cup_j N(\Upsilon_j)$ is diffeomorphic with
 $\partial F\setminus
 \cup_j N(\Upsilon_j)$. But, the former space can be identified via $\Phi$ by $\mathfrak{S}^3\setminus \cup_i
 (N_i\cup N_{\sigma(i)})$.
 This is the space in the first parenthesis of (\ref{eq:mb}).

 The second one  is a union of spaces of type $X_j:=\partial F\cap N(\Upsilon_j)$,
 which fibres over $\Upsilon_j\simeq S^1$.  The fibre of the fibration is the
 Milnor fibre of the corresponding transversal plane curve singularity (of $\Sigma_j$).
 Since the transversal type is $A_1$, this fibre is $F_j:=S^1\times I$. The monodromy of
 the fibration is the so
 called {\it geometric vertical monodromy } of the transversal type, it is orientation
 preserving  self-diffeomorphism of $ S^1 \times I $. If it does not permute
 the two components of $\partial F_j$ then it preserves the orientation of
 $I$, hence of $S^1$ too, hence up to isotopy it is the identity. If it
 permutes the components of $\partial F_j$ then
up to isotopy it is
 $ (\alpha, t) \to (\bar{ \alpha}, - t)$, where $ (\alpha, t ) \in S^1 \times I $.
  %and $ \bar{ \alpha} $ denotes the image of $ \alpha \in S^1 $ along an axial reflection (i.e. complex conjugation).
The two types of vertical monodromies provide the two choices of  $ X_j $ in formula (\ref{eq:xj}), cf. description (1) of $Y$  in
subsection \ref{ss:cY}.
\end{proof}

\subsection{Preliminary discussion regarding the gluing}\label{ss:gluinguj}
Our next aim is to describe the gluing functions $ \phi_j $. In both cases two tori must be glued:  if $ i\neq\sigma(i)$
then basically one should identify  $ \partial N_i $ and $ -\partial N_{\sigma (i)}$,
 otherwise $\partial Y$ and $\partial N_i$.
Up to diffeotopy an orientation reversing diffeomorphism between tori is given by an invertible
 $2 \times 2 $ matrix over $ \Z $ with determinant $-1$.
 It turns out that in our cases all these gluing matrices have the form
\begin{equation}\label{eq:glue}
\left( \begin{array}{cc}
  -1 & n_j\\
  0 & 1 \\
 \end{array} \right)
\end{equation}
 \noindent hence its only relevant entry is
 the off--diagonal one.

 This integer will be determined by a newly
 introduced invariant, the {\it vertical index},  associated with each
$j\in J$.
 This is done using a special germ $H:(\C^3,0)\to (\C,0)$,
which will have a double role.
First, it provides some kind of framing along $\Sigma\setminus \{0\}$,
and also helps to
identify generators from the boundaries of  $K\setminus \cup_j N(\Upsilon_j)$ and  $\mathfrak{S}^3\setminus \cup_i (N_i\cup N_{\sigma(i)})$
respectively (constructed in two different levels: in the target and in the source of $\Phi$).

We will use three parametrizations of $ \partial N_i \approx S^1 \times S^1 $ with the same meridian but different longitudes.
The topological longitude is the usual knot--theoretical Seifert--framing of
 $ L_i \subset S^3 $. The \emph{resolution longitude}
 is determined via  a good embedded resolution of $ (D, 0) \subset ( \C^2, 0 ) $,
  it creates the bridge with the decorations and the combinatorics of
  the resolution graph. Finally,
the \emph{sectional longitude} depends on $ H$ and it will be used for describing the gluing. In fact, the sectional longitude allows us to compare the source and the target of $ \Phi $: being defined by the geometry of $H$, the function $H$ and
its pull-back $ \Phi^* H $  plays the role of transportation of the invariants from
$\C^3 $ level to $\C^2$ level.

\subsection{The local form of $f$ along $ \Sigma $}\label{ss:local}
In the sequel we will use the notation $\Sigma^*=\Sigma\setminus \{0\}$ and $\Sigma^*_j=\Sigma_j\setminus \{0\}$.
 Recall that in a small neighborhood (in $\C^3$) of any point $p\in \Sigma^*$ the
 space $(X,0)$ has two local
 components, both smooth and  intersecting each other transversally.

A more precise local description along $\Sigma^*$  is the following.
Let us fix a point $p_0\in \Sigma_j^*$ and let $U_0$ be a small
neighborhood of $p_0$ in $\C^3$. In $U_0$ the function $f$ is a product $f_1\cdot f_2$,
where both $f_n$ are holomorphic, $\{f_n=0\}$ are smooth
 and intersect each other transversally.
 (The intersection is $\Sigma^*\cap U_0$; later the fact that at $p_0$ the local parametrization  of
  $\Sigma^*\cap U_0$ together with $f_1$ and $f_2$ might serve as local coordinates will be exploited further.)

 $f_1$ and $f_2$  are well--defined up to a multiplication
 by an invertible holomorphic  function $\iota$ of
 $U_0$; that is, $(f_1,f_2)$ can be replaced by $(\iota  f_1, \iota^{-1}f_2)$.
% and if  $i=\sigma(i)$ then $(f_1,f_2)$ by $(f_2,f_1)$.
 At any point $p\in \Sigma^*_j\cap U_0$ the linear term of $f_n$, $n\in\{1,2\}$,
(say, in the Taylor expansion) is $T_1(f_n)= \sum _{k=1}^3 u_{nk}(p)(x_k-p_k)$, where
$(x_1,x_2,x_3)$ are the fixed coordinates of $(\C^3,0)$.
Let us code this in the non--zero vectors $u_n(p):=(u_{n1}(p),u_{n2}(p), u_{n3}(p))$.
Hence, at any $p\in \Sigma^*_j\cap U_0$ we have two vectors
$u_1(p)$ and $u_{2}(p)$ well--defined up to
multiplication by $\iota|_{\Sigma_j^*\cap U_0}$ (in the sense described above).
%and in the case  $i=\sigma(i)$ up to a permutation.
Their classes $ [u_n(p)]\in
\C \mathbb {P} ^2$ for $p\in \Sigma^*_j\cap U_0 $
 are independent of the $\iota$--ambiguity, hence are well--defined elements.
%up to permutation whenever $i=\sigma(i)$.
In particular, they determine a global pair of elements
$[u_1(p)]$ and  $[u_2(p)]\in
\C \mathbb {P} ^2$ for $p\in \Sigma^*_j $, well--defined whenever $i\not=\sigma(i)$, and
well--defined up to permutation whenever $i=\sigma(i)$.
%(In other words, if $i=\sigma(i)$, then $\Sigma^*_j\ni p\mapsto \{[u_1(p)],
%[u_2(p)]\}$ is a two--valued function.)

In fact, we can do even more: there exists a splitting of $f$ into product $f_1\cdot f_2$ along $\Sigma^*_j$ without any invertible element ambiguity (but preserving the
permutation ambiguity whenever $i=\sigma(i)$).

Indeed, assume that we are in the trivial covering ($i\not=\sigma(i)$) case, and
let us cover $\Sigma_j^*$ by small discs $\{U_\alpha\}_\alpha$ such that
on each $U_{\alpha}$ we can fix a splitting $f(p)=f_{1,\alpha}(p)
\cdot f_{2,\alpha}(p)$, $p\in U_\alpha$. For any intersection $U_{\alpha\beta}=
U_{\alpha}\cap U_{\beta}$ the two splittings can be compared: we define
$\iota_{\alpha\beta}\in \calO(U_{\alpha\beta})^*$ by $f_{1,\alpha}|_{U_{\alpha\beta}}=\iota_{\alpha\beta}\cdot
f_{1,\beta}|_{U_{\alpha\beta}}$. From this definition follows that
$\{\iota_{\alpha\beta}\}_{\alpha, \beta}$ form a \v{C}ech 1--cocycle.
\begin{lemma}
$H^1(\Sigma^*_j,\calO_{\Sigma^*_j}^*)=0$.
\end{lemma}
\begin{proof}
From the exponential exact sequence $0\to \Z \to \calO\to \calO^*\to 0$ over $\Sigma_j^*$,
we get that it is enough to prove the vanishing
$H^1(\Sigma^*_j,\calO_{\Sigma^*_j})=0$, a fact which follows from Cartan's Theorem, since $\Sigma_j^*$ is Stein.
\end{proof}

Since $H^1(\Sigma^*_j,\calO_{\Sigma^*_j}^*)=0$, the cocycle
$\{\iota_{\alpha\beta}\}_{\alpha, \beta}$  is a coboundary.
This means that we can find invertible functions $\iota_\alpha $ on
each $U_\alpha$ such that on $U_{\alpha\beta}$ one has
$\iota_{\alpha}|_{U_{\alpha\beta}}=\iota_{\alpha\beta}\cdot
\iota_{\beta}|_{U_{\alpha\beta}}$. This means that  the local functions
 $\widetilde{f}_{1, \alpha}:=f_{1,\alpha }\cdot \iota_{\alpha}^{-1}$,
  $\widetilde{f}_{2, \alpha}:=f_{2,\alpha }\cdot \iota_{\alpha}$ on $U_{\alpha}$
  provide a splitting (that is, $f(p)=\widetilde{f}_{1,\alpha}(p)
\cdot \widetilde{f}_{2,\alpha}(p)$, $p\in U_\alpha$), but in this new situation
the local splittings glue globally: $\widetilde{f}_{1,\alpha}|_{U_{\alpha\beta}}=
\widetilde{f}_{1,\beta}|_{U_{\alpha\beta}}$.
If $i=\sigma(i)$ then we repeat the proof on $D_i$.

 \subsection{The special germ $H$} Next, we treat the `aid'--germ $H$.

\begin{defn}\label{def:H} Let us fix $\Phi, \ f, \ \Sigma$ as above.
 A germ $ H: ( \C^3, 0) \to ( \C ,0) $ is called \emph{transversal section} along $ \Sigma $ if
 $ \Sigma \subset H^{-1} (0) $, $H^{-1}(0)$ at any point of $\Sigma^*$ is smooth and intersects both local
 components of $(X,0)$ transversally.
\end{defn}

We claim that  transversal sections  always exist.

\begin{prop}\labelpar{pr:H}
 There exist  complex numbers  $a_1, a_2, a_3 $ such that
 $H= a_1 \partial_{x_1} f + a_2 \partial_{x_2} f + a_3 \partial_{x_3} f $
 is a transversal section.
 %Moreover, we can even assume that $H=0$ intersects $f=0$ transversally off $\Sigma$.
\end{prop}

\begin{proof} For such $H$ one has
%\begin{equation}\labelpar{eq:kifh}
$  T_1(H)  = \sum_{k=1}^3 a_k u_{1k}\cdot T_1(f_2)+
  \sum_{k=1}^3 a_k u_{2k}\cdot T_1(f_1)$.
%\end{equation}
In particular, we have to show that for certain coefficients $\{a_k\}_k$
the expressions $ \beta_n(p)=\sum_{k=1}^3 a_k u_{nk} ( p )$ (for $n\in\{1,2\}$)
 have no zeros for $p\in \Sigma^*$.
%Consider the holomorphic map
%\begin{equation*}
%  \tilde{u}_n: \Sigma_j^* \to \C \mathbb {P} ^2\,, \ \
%  \tilde{u}_n (p) = [u_{n}(p)].
%\end{equation*}

If $\gamma : X^N\to X$ is the Nash transform of $X$, then $\gamma^{-1}(0)$
is the set of limits of tangent spaces of $X\setminus \{0\}$, it is an
algebraic set of  $\C \mathbb {P} ^2$ of dimension $\leq 1$
(for details see e.g. \cite{LeTeissier} and references therein).
The existence of Nash transform guarantees that
$ [\ell_n]=\lim_{p\to 0}[u_n(p)] \in \C\mathbb{P}^2$  exist.
Indeed, the set $\{[l_1],[l_2]\}$ is the intersection of
$\gamma^{-1}(0)$ with the strict transform of $\Sigma_j$.
Then
let $[a_1:a_2:a_3]$ be generic such that
$\sum_ka_k\ell_{nk}\not=0$ for $n\in\{1,2\}$.
With this choice  $\sum_{k=1}^3 a_k u_{nk}(p) \not=0$ for $p\not=0$ and
in a small representative of $\Sigma_j$.
\end{proof}

Fix  again
$j=\{i, \sigma(i)\}\in J $, and let  $ p: (\C, 0) \to (\Sigma_j,0)\subset (\C^3, 0) $,
$\tau\mapsto p(\tau)$,  be a
parametrization (normalization) of  $ \Sigma_j $.
  For any point  $p_0=p(\tau_0)$  and neighborhood  $\Sigma^*_j\cap U_0 \ni p(\tau)$ the
discussion from the second paragraph of \ref{ss:local}   can be repeated, in particular
we have the holomorphic
 vectors $u_n(p(\tau))$ ($n\in\{1,2\}$) (with the choice ambiguities described there).
Additionally, choose some $H$ as in Definition \ref{def:H}.
The assumption regarding $H$   guarantees that
$T_1(H)(p(\tau))=\beta_1(\tau)T_1(f_1)(p(\tau))+\beta_2(\tau)T_1(f_2)(p(\tau)) $
for some holomorphic functions $\beta_1$ and $\beta_2$ on $p^{-1}(\Sigma_j^*\cap U_0)$. If we replace
$(u_1,u_2)$ by $(\iota u_1,\iota^{-1}u_2)$ then $(\beta_1,\beta_2)$ will be replaced by
$(\iota^{-1}\beta_1,\iota\beta_2)$,
hence the product $\beta_1\beta_2$ is independent of all the $\iota$ and permutation ambiguities. It is a holomorphic function
on $p^{-1}(\Sigma_j^*\cap U_0)$ depending only on the equations $f$ and $H$. This uniqueness also guarantees
that taking different points of $\Sigma^*_j$ and repeating the construction, the output glues to a unique
holomorphic function $\mathfrak{b}_j(\tau)$ on a small punctures disc of $(\C,0)$.
Usually $\mathfrak{b}_j(\tau)$ has no analytic extension to the origin, however
one has the following.
\begin{lemma}\label{lem:Laurent} $\mathfrak{b}_j(\tau)$  is
a  Laurent series on $(\C,0)$.
\end{lemma}
\begin{proof}
The finiteness of the poles follows e.g. from the homological identities from
Theorem \ref{th:mainth}, or from Corollary \ref{cor:identities}  combined with (\ref{eq:alpha}).
\end{proof}

We wish to emphasize that $ \mathfrak{b}_j(\tau) $ and its pole order usually depends on the choice of
$H$.

\begin{defn}\label{de:vert} Let $a \tau^{\mathfrak{v}_j} $ (for some $a\in\C^*$) be the non--zero
monomial with smallest power of
$\tau $ in the Laurent  series of $ \mathfrak{b}_j(\tau) $.
% (if such power does not exists, then we say that
%$\mathfrak{v}_j=-\infty $, however in ??????????????? we will prove that
%$\mathfrak{v}_j\not=-\infty $).
The integer $ \mathfrak{v}_j $ is called the \emph{vertical index
of $ f$ along $ \Sigma_j $ with respect to $ H$} (or, the $H$--vertical index).
\end{defn}

\subsection{Computation of the gluing functions $ \phi_j $}

We fix a transversal section  $H$ (cf. \ref{def:H}).
%which intersects $X\setminus \Sigma$ transversally (cf. \ref{pr:H}).
Then the divisor  $ \Phi^{*} (H) $ is $ H\circ \Phi=d\cdot d_\sharp $ for some (not necessarily reduced)
germ $d_\sharp:(\C^2,0)\to (\C,0)$ (such that $d$ and $d_\sharp $ have no common components).
Let
%$d'=\prod_n (d'_n)^{\alpha_n}$ be the irreducible decomposition of  $d'$,
$(D_\sharp,0)$ be the (non--reduced) divisor associated with $d_\sharp$, and
$N(L_\sharp)$ be a small tubular neighborhood of the reduced link
$L_\sharp:={\rm red}(D_\sharp)\cap \mathfrak{S}^3$ in $\mathfrak{S}^3$.

Let the Milnor fibre of $ \Phi^{*} (H) $ (in $\mathfrak{S}^3$)  be
\[ F_{\Phi^{*} (H)} = \{ (s, t) \in \partial \mathfrak{B}_\epsilon=\mathfrak{S}^3
 \ | \  H ( \Phi  (s, t) )>0 \}. \]
Let $ \Lambda_i $ and $ \Lambda_\sharp $
denote  the components of the oriented intersection
$F_{\Phi^{*} (H)}\cap\, \partial N_i$ and $F_{\Phi^{*} (H)}\cap \partial N(L_\sharp)$
with the tubular neighborhood boundaries
  of  $ L_i $ and $L_\sharp$ respectively.
 ($\Lambda_\sharp$ might have several components, in this notation we collect all of them.)

 Furthermore,
let $ M_i \subset \partial N_i $ be an oriented
 meridian of $ \partial N_i $ such that $ \lk (M_i, L_i) = 1 $
  %and $M_i$ bounds a disc in $ N_i$.
and fix  also (the oriented Seifert framing of $D_i $)
 $ L'_i \subset \partial N_i $ with  $ \lk (L'_i, L_i) =0 $.
 (Here the linking numbers are considered in oriented 3--sphere $\mathfrak{S}^3$.)

\begin{defn} We call $ L'_i $ the \emph{topological longitude}
 of the torus $ \partial N_i $, while
$ \Lambda_i $ the \emph{sectional longitude} of $\partial N_i$
 associated with the transversal section  $H$.
\end{defn}
Clearly we have the following facts (where $ [\cdot ] $ denotes the corresponding homology class)
\begin{equation}
\begin{split}
(a) & \ \ H_1 (N_i,\Z) \cong \Z\langle [\Lambda_i]\rangle  \cong\Z\langle [L'_i] \rangle, \\
(b) & \ \  H_1 ( \partial N_i, \Z ) \cong \Z \langle [ \Lambda_i ] \rangle \oplus \Z \langle [ M_i ] \rangle \cong
 \Z \langle [ L'_i ] \rangle \oplus \Z \langle [ M_i ] \rangle.\\
\end{split}
\end{equation}
We want to express $  [ \Lambda_i ] $ in terms of $ [ M_i ] $ and $ [L'_i ] $.

\begin{lemma}\label{lem:bases}
 Define
 $\lambda_i = - \Sigma_{k \not=i } D_k \cdot D_i -D_\sharp \cdot D_i$,
 where $ C_1 \cdot C_2 $ denotes the intersection multiplicity of
 $ (C_1, 0) $ and $ (C_2, 0)  $ at $0\in \C^2$.
 Then $ [ \Lambda_i ] = [L'_i ] + \lambda_i \cdot [M_i ] $
  in $ H_1 ( \partial N_i, \Z )$.
\end{lemma}

\begin{proof} First note that
$ [\Lambda_i ] = [L'_i ] + \lk ( \Lambda_i, L_i ) \cdot [M_i ] $.
Write $F'_{\Phi^{*} (H)}:=F_{\Phi^{*} (H)}\setminus (\mbox{int}(\cup_i N_i\cup N(L_\sharp)))$.
Then
$0=  \lk ( \partial F'_{\Phi^{*} (H)} , L_i )=\sum_k \lk (\Lambda_k,L_i)+
\lk(\Lambda_\sharp,L_i)=\lk(\Lambda_i,L_i)+\sum_{k\not=i}D_k\cdot D_i+D_\sharp\cdot D_i= \lk(\Lambda_i,L_i)-\lambda_i$.
\end{proof}

\begin{thm}\label{th:mainth} For any $j=\{i,\sigma(i)\}\in J$, the gluing functions $ \phi_j $ from
Proposition~\ref{pr:gl}
is characterized up to homotopy by the following identities

{\bf Case 1:} \ $ i \neq \sigma (i) $. Identify the homology groups $  H_1 (S^1\times S^1 \times \{ 1 \}, \Z ) $,
$ H_1 (S^1\times S^ 1\times \{ -1 \}, \Z ) $ and $  H_1 (S^1\times S^1 , \Z ) $ via the natural homotopies
$$S^1\times S^1\times\{-1\}\stackrel{h}\sim S^1\times S^1\times [-1,1]\stackrel{h}\sim
S^1\times S^1\times\{1\}.$$ Then in this homology group one has
\[
 \phi_{j*} ([M_i]) = - \phi_{j*} ([M_{\sigma(i)}]) \ \  \mbox{ and }  \]
 \[ \phi_{j*} ([\Lambda_i])  = \phi_{j*} ([\Lambda_{\sigma(i)}] +  \mathfrak{v}_j
 \cdot \phi_{j*} [M_{\sigma(i)}]). \]
%where $ \phi_{i*} ([M_i]) $ and $ \phi_{i*} ([\Lambda_i]) $ are in ,
%$ \phi_{i*} ([M_{\sigma(i)}]) $ and $ \phi_{i*} ([\Lambda_{\sigma(i)}]) $ are in
%, and both homology groups are naturally isomorphic with .

{\bf  Case 2:} \ $ i = \sigma (i) $.
 \[ \phi_{j*} ([M_i] ) = -m \ \ \mbox{ and } \] %\phi_{j*} ([M_i]) = - \phi_{j*} ([M_{\sigma(i)}])
 \[ \phi_{j*} ([\Lambda_i] ) = c+ \mathfrak{v}_j \cdot m, \]
 where $m $ and $c $ are the two generators of $ H_1( \partial Y, \Z)$, see \ref{ss:homY} (especially
 (\ref{eq:pY}) and (\ref{eq:cmu})).
\end{thm}

%In other words, $ \partial N_i $ and $ \partial N_{ \sigma(i) } $ are glued together to form $ \partial F $ such
%that the meridian goes to meridian with opposite orientation and the algebraic longitude goes to algebraic longitude
%in orientation preserving way. In the second case, $ \partial N_i $ and $ \partial Y $ are glued together to form
%$ \partial F $ such that the meridian goes to $  \mu $ with opposite orientation and the algebraic longitude goes to
%$  \lambda^2 $ in orientation preserving way.

\begin{proof}
Recall that $\tau\mapsto p(\tau)$ is the normalization of $\Sigma_j$.
At any point $p(\tau_0)$
of $\Sigma^*_j$ one can consider $\tau$ as a local complex coordinate, which
can be completed with two other  local complex coordinates $(x,y)$ (local coordinates in a
transversal slice of $\Sigma_j$ at $p(\tau_0)$)   such that $(\tau,x,y)$
form a local coordinate system of $(\C^3,p(\tau))$, and  locally
$f(x,y)=xy $.
These two local coordinates correspond to a splitting $f=f_1\cdot f_2 $ of $f$.
According to the discussion from \ref{ss:local}, the splitting of $f$ can be done globally
along the whole $\Sigma_j^*$, hence these coordinates $(x,y)$ (corresponding to the
components $f_1$ and $f_2$) can also be chosen globally along $\Sigma_j^*$ (with the permutation ambiguity
whenever $i=\sigma(i)$, a fact which will be handled below).

 Furthermore, in these coordinates, $T_1 H=\beta_1(\tau) x+\beta_2(\tau) y$.
Since $\beta_1(\tau)$ have no zeros and poles in the small representative of $\Sigma_j^*$, $(\tau, \beta_1x, \beta_1^{-1}y)$ are also local coordinates, and in these coordinates the equations
 transform into   $f(x,y)=xy$ and $T_1 H(x,y)=x+\mathfrak{b}_j(\tau) y$.

If we concentrate on the points of  $\Upsilon_j=\Sigma_j\cap S^5_\epsilon$, and its neighbourhood in $S^5_\epsilon$, then similarly as above, we have the real
coordinate $\tau \in \Upsilon_j$, and the two complex (transversal) local
coordinates $(x,y)$, with equations  $f(x,y)=xy$ and $T_1 H(x,y)=x+\mathfrak{b}_j(\tau) y$ as before.

{\bf Case 1.}
The above  local description  globalises as follows
(compare also with the first part of the proof of Proposition \ref{pr:gl}).
 The space $\partial F\cap N(\Upsilon_j)$
has a product decomposition
$\Upsilon_j\times F_j=S^1\times S^1 \times I$, where $\Upsilon_j=S^1$
is the parameter space of $\tau$,
and $F_j$ is the local Milnor fibre $F_j=\{xy=\delta\}\cap B^4_\varepsilon$,
$0 <\delta\ll \varepsilon\ll \epsilon $,
diffeomorphic to $S^1\times I$.
In other words, $\partial F\cap N(\Upsilon_j)$ is the space
$\{(\tau,x,y)\in S^1\times B^4_\varepsilon\,:\, xy=\delta\}$,
$0 <\delta\ll \varepsilon\ll \epsilon$ (here we will use the same $\tau$ notation for the
parameter of $S^1$). Note that for $\delta\ll \varepsilon$,
the boundary of $F_j$ is `very close' to the two circles
$\{|x|=\varepsilon, \, y=0\}$ and $\{x=0, \, |y|=\varepsilon\}$ of $\partial B^4_\varepsilon$.
Using isotopy in the neighborhoods of these two circles,
$\partial F_j$ can be identified with these two circles
(similarly as we identify via Ehresmann's fibration theorem the boundary of the
Milnor fibre of an isolated singularity with the link),
and in order to
simplify the presentation, we will make this identification.
Hence, the boundary components of $\partial F\cap N(\Upsilon_j)$ can be identified (by isotopy in $ \partial N( \Upsilon_j ) $) with
 $\partial _i:=\{(\tau,x,y)\,:\,
\tau\in S^1,\, |x|=\varepsilon, \, y=0\}$ and $\partial_{\sigma(i)}:=
\{(\tau,x,y)\,:\, \tau\in S^1,\, x=0, \, |y|=\varepsilon\}$.
(The choice of indices $i$ and $\sigma(i)$ is arbitrary and symmetric.)
These two tori are identified homologically  since they are boundary
components of $S^1\times F_j$.
In $\partial_i$ we have the meridian
$\tilde{M}_i=\{(\tau,x,y)\,:\, \tau=1, \,|x|=\varepsilon , \, y=0\}$, while in
$\partial_{\sigma(i)}$ we have the meridian
$\tilde{M}_{\sigma(i)}=\{(\tau,x,y)\,:\, \tau=1, \,x=0, \, |y|=\varepsilon\}$
(both naturally oriented as the complex unit circle).
Since $\partial (\{\tau=1\}\times F_j)
=\tilde{M}_i\cup -\tilde{M}_{\sigma(i)}$, the first wished
identity follows.

Now, we would like to study the intersection curve of $ \partial_i $ and the Milnor fibre of $ H $ associated with a positive argument, that is
$ \tilde{\Lambda}_i :=  \partial_i \cap \{ H > 0 \} $. This curve is homotopic with $ \partial_i \cap \{ T_1 H=x+\mathfrak{b}_j(\tau) y > 0 \} $ and also with $ \partial_i \cap \{ x+\tau^{\mathfrak{v}_j}y > 0 \}$ in $ \partial_i $, cf. the definition \ref{de:vert}. Similarly, $ \tilde{\Lambda}_{\sigma (i)} :=  \partial_{\sigma (i)} \cap \{ H > 0 \} $ is homotopic with  $ \partial_{\sigma (i)} \cap \{ x+\tau^{\mathfrak{v}_j}y > 0 \}$ in $ \partial_{\sigma (i)} $.

%\marginpar{Szerintem $ \mathfrak{b} $ rossz jeloles, alig ter el
 %$ \mathfrak{v} $ tol. Talan $ \mathfrak{B} $?}

%$H(\tau,x,y)=x+\tau^{\mathfrak{v}_j}y$.
%Then $H$
%cuts (via its Milnor fibre associated with a positive argument) in
%\marginpar{jeloljuk ezeket $\tilde{M}$ es $\tilde{\Lambda}$-kel??? OK!!!}
%$\partial _i$ the circle
Thus $\tilde{\Lambda}_i $ is homotopic with $
\{(\tau ,x,y)\,:\, \tau\in S^1, \, |x|=\varepsilon,\, y=0\}\cap \{ x+\tau^{\mathfrak{v}_j}y>0\}=
\{(\tau ,x,y)\,:\, \tau\in S^1,\, x=\varepsilon, \, y=0\}$ and
 $\tilde{\Lambda}_{\sigma(i)} $ with $
\{(\tau ,x,y)\,:\, \tau\in S^1,\, x=0, \, y=\varepsilon \tau^{-\mathfrak{v}_j}\}$.
%\marginpar{kicsreltuk $b_j$-t $v_j$-re}
Hence homologically $\tilde{\Lambda}_{\sigma(i)}+\mathfrak{v}_j \tilde{M}_{\sigma(i)}$
is represented by  the circle
$\{(\tau ,x,y)\,:\, \tau\in S^1,\, x=0, \, y=\varepsilon\}$, which is
homologous in $S^1\times F_j$ with $\tilde{\Lambda}_i$. This is the second identity.
Obviously, these identities can be transferred from  the boundary
$\partial _i\cup \partial _{\sigma(i)} $ of $S^1\times F_j$
into similar identities in $\partial N_i\cup \partial N_{\sigma(i)}$, via the diagramm, where all the maps are orientation preserving
 diffeomorphisms (cf. the proof of \ref{pr:gl}):
\begin{equation}\labelpar{eq:dia}
 \begin{array}{cccc}
  \Phi: & \mathfrak{S}^3 \setminus \cup_i N_i & \to  & K \setminus \cup_j N( \Upsilon_j ) \\
          & \partial N_i \sqcup \partial N_{ \sigma(i) } & \to & \partial_i \sqcup \partial_{ \sigma(i) } \\
          & \Lambda_i    , \Lambda_{ \sigma(i) }                                & \to & \tilde{ \Lambda}_i,  \tilde{\Lambda}_{ \sigma(i) } \\
          & M_i    , M_{ \sigma(i) }                                & \to & \tilde{ M}_i,  \tilde{M}_{ \sigma(i)}. \\
 \end{array}
\end{equation}

\vspace{1mm}

\noindent {\bf Case 2.}  We use similar notations and conventions as in Case 1.
%\marginpar{VERIFY}
Let us parametrize $\Upsilon_j$ as $\tau=e^{2is}$,
$s\in[0,\pi]$. Then $ \partial F\cap N(\Upsilon_j)$ is
$([0,\pi]\times F_j)/\sim $,
%\marginpar{le kellene irni ezt az azonositast pontosabban}
where by $\sim $ we identify $(0,x,y)\sim (\pi,y,x)$ for all $(x,y)\in F_j$.
Let us parametrize $ F_j $ as $ x= \sqrt{ \delta} r e^{it} $ and $ y = \sqrt{ \delta} r^{-1} e^{-it} $, where $ t \in [- \pi , \pi] $ and
$ r \in [r_0 , r_1] $ such that $ r_0 r_1=1$ and $\delta (r_0^2 + r_1^2) = \varepsilon^2 $. Denote $ z = \log_{r_1} r $.
% $ r \in [\sqrt{\delta}/ \epsilon , \epsilon / \sqrt{\delta}] $. Denote $ z = \log_{\epsilon / \sqrt{\delta}} r $.
Then we can parametrize $ \partial F\cap N(\Upsilon_j)$ by $ (s, t, z) $, thus
$  \partial F\cap N(\Upsilon_j) $ is just $([0,\pi]\times [- \pi, \pi ] \times [-1, 1])/\sim $,
where by $\sim $ we identify
$  (0, t, z) \sim ( \pi, -t, -z) \mbox{ and } (s, -\pi, z) \sim (s, \pi, z) $.
 We regard this as parametrization of
 $  \partial F\cap N(\Upsilon_j) $ by $ Y $, cf.  \ref{ss:homY}.

%\marginpar{$z$ helyett vmi mas? nem zavaro itt a $z$?
%am ha csereljuk, akkor az $Y$ leirasaban is cserelni kell}

Set
\[ \tilde{M}_{j1} = \{ ( \tau = 1,\ x= \sqrt{ \delta} r_1 e^{it},\ y = \sqrt{ \delta} r_0 e^{-it}  ) \} \mbox{ and }\]
\[ \tilde{M}_{j2}  = \{ ( \tau =1,\ x= \sqrt{ \delta} r_0 e^{-it},\ y = \sqrt{ \delta} r_1 e^{it}  ) \}. \]
The  are the two oriented meridians of $ \partial F\cap N(\Upsilon_j)$,
parametrized by $ t \in [ - \pi, \pi ] $. In terms of  $ (s, t, z) $ they are
 \[ \tilde{M}_{j1}=\{(s=0,\, t , \, z=1)\} = \{(s=\pi,-t, \, z=-1)\} \subset \partial Y \mbox{ and } \]
\[ \tilde{M}_{j2}=\{(s=0,\, -t , \, z=-1)\} = \{(s=\pi,t, \, z=1)\} \subset \partial Y \mbox{,} \]
thus with the notations of \ref{ss:homY}, $\tilde{M}_{j1} = -m'$ and $ \tilde{M}_{j2}= -m $.
%\marginpar{Irjunk at vmit? Pl. a kockan $m$ es $m'$ ha masik helyre
%lenne barajzolva mar jo lenne, nem minuszba lennenek.}

Similarly as in Case 1, by an isotopy in $ \partial N( \Upsilon_j ) $ the boundary $ \partial_j $ of $ \partial F\cap N(\Upsilon_j)$ can be identifyed with $ \partial_{j1} \cup \partial_{j2} $, where
\[\partial _{j1}:=\{(s,x,y)\,:\,
s \in [0, \pi],\, |x|=\varepsilon, \, y=0\} \mbox{ and }\]
\[ \partial_{j2}:=
\{(s,x,y)\,:\, s \in [0, \pi],\, x=0, \, |y|=\varepsilon\} \mbox{.} \]
The two parts of $ \partial_j $ are glued together along the image of the oriented meridians
\[ \tilde{M}_{j1}=\{(s=0,\, x= \varepsilon e^{it}, \, y=0)\} = \{(s=\pi,\, x=0, \, y=\varepsilon e^{it})\}  \mbox{ and} \]
\[ \tilde{M}_{j2} = \{(s=0,\, x=0,\, y= \varepsilon e^{it})\} = \{(s= \pi,\, x= \varepsilon e^{it}, \, y=0)\} \mbox{.} \]
% The isotopy of $ \partial_j \subset S^5 $ moves the point
% ( \tau,\ x= \sqrt{ \delta} r_1 e^{it},\ y = \sqrt{ \delta} r_0 e^{-it}  )$ to
% $ ( \tau , \ x= \varepsilon e^{it}, \ y=0) $ and
% $( \tau,\ x= \sqrt{ \delta} r_0 e^{it},\ y = \sqrt{ \delta} r_1 e^{-it}  )$ to
%$ ( \tau , \ x= 0, \ y= \varepsilon e^{-it}) $.
% Taking the pre-image of the meridians, one can express them with the $ (s, t, z) $-parametrization of $ Y $, and get the parametrizations by $ t \in [0, \pi] $ as
% $ \tilde{M}_{j1}=\{(s=0,\, t , \, z=1)\} = \{(s=\pi,-t, \, z=-1)\} \subset \partial Y$ and
% $ \tilde{M}_{j2}=\{(s=0,\, t , \, z=-1)\} = \{(s=\pi,-t, \, z=1)\} \subset \partial Y$.
(See also the description in \ref{ss:homY}.)
Since the oriented boundary $ \partial ( \{ \tau = 1 \} \times F_j ) $ is $ \tilde{M}_{j1} \sqcup \tilde{M}_{j2} $, $ \tilde{M}_{j1}$ is homologous with $ - \tilde{M}_{j2} $ in $ F_j \subset Y$. On the other hand $ \tilde{M}_{j1} $ is homologous with $ \tilde{M}_{j2} $ in $ \partial Y \subset Y $, thus $ [\tilde{M}_{j1}]= [\tilde{M}_{j2}]$ is an order--$2$ element in $H_1 ( Y, \Z ) $.

Consider the closed curve $C$ obtained as union of
$C_{1}$ and $C_{2}$, where
\[C_{1}=\{(s,x=\varepsilon, y=0)\,:\, s\in [0,\pi]\} = \{(s,t=0, z=1)\,:\, s\in [0,\pi]\}  \mbox{, and} \]
  \[ C_{2}=\{(s,x=0, y= \varepsilon)\,:\, s\in [0,\pi]\} = \{(s, t=0, z=-1) \,:\, s\in [0,\pi]\}\mbox{.} \]
  Note that
$C_{1}$ connects the points $A_1=(s=0, x=\varepsilon,y=0)=(s=0, t=0, z=1)$ with
$B_1=(s= \pi, x=\varepsilon, y=0)= (s=\pi, t=0, z=1)$,
while
$C_{2}$ connects the points $A_2=(s=0,x=0,y=\varepsilon)= (s=0, t=0, z=-1)$ with
$B_2=(s= \pi, x= 0,y=\varepsilon)= (s= \pi, t= 0, z=-1)$. Since $A_1\sim B_2$ and $A_2\sim B_1$,
they form a closed curve. Note that $ [C]=[c] $ in $ H_1 ( \partial Y, \Z ) $, see \ref{ss:homY}.

Similarly as in Case 1, the function $H$ can be replaced by
$x+\tau^{\mathfrak{v}_j}y$, hence its level set associated with a positive value determines the curve
$ \tilde{\Lambda}_j = \{ x+\tau^{\mathfrak{v}_j}y >0 \} \cap \partial_j $. This consists of two parts, $\tilde{\Lambda}_{j1}$ and $\tilde{\Lambda}_{j2}$, where
\[ \tilde{\Lambda}_{j1}=
\{(s,\, x=\varepsilon,\, y=0)\,:\, s\in [0,\pi]\}
= \{(s,t=0, z=1)\,:\, s\in [0,\pi]\} \subset \partial_{j1},  \]
while
$ \tilde{\Lambda}_{j2}=
\{(s,\, x=0,\, y=\varepsilon e^{-2is\mathfrak{v}_j})\, :\, s\in[0,\pi]\} $ equals
 \[ \{(s,t=2 \mathfrak{v}_j s \ (\mbox{mod } [- \pi, \pi ]), z=-1)\,:\, s\in [0,\pi]\} \subset \partial_{j2}. \]
$ \tilde{\Lambda}_{jn}$ has the same end--points as $C_{n}$, hence
%\marginpar{nekem az elojelek csak nem vilagosak, nekem meg a 2-es sem!
%ezt itt nem ertem}
$\tilde{\Lambda}_{j1}$ and $\tilde{\Lambda}_{j2}$ form together a closed curve, as we expect.
Furthermore, $\tilde{\Lambda}_j+\mathfrak{v}_j \tilde{M}_{j2}$ is homologous in $\partial Y$ with $C$.

The source and the target are connected by the restriction of $ \Phi $, which gives the orientation preserving  diffeomorphisms:
\begin{equation}\labelpar{eq:dia2}
 \begin{array}{cccc}
  \Phi: & \mathfrak{S}^3 \setminus \cup_i N_i & \to  & K \setminus \cup_j N( \Upsilon_j ) \\
          & \partial N_i  & \to & \partial_j  \\
          & \Lambda_i      & \to & \tilde{ \Lambda}_j \\
          & M_i    & \to & \tilde{ M}_{j1} \mbox{ homologous with } \tilde{ M}_{j2} \\
 \end{array}
\end{equation}

Since $C$ identifies with $c$ and $ \tilde{M}_{j1}$ and $ \tilde{M}_{j2}$ with $-m$,
$\tilde{\Lambda}_i=c+\mathfrak{v}_j\cdot m$ follows.
\end{proof}

\subsection{The $H$--independent description of the gluing. The `vertical index'}\label{ss:topverindex}\

Recall that the sectional longitudes $ \Lambda_i $  and the corresponding
$H$--vertical indexes  $ \mathfrak{v}_j$ depend on the choice of $H$.
 The goal of this paragraph is to replace $(M_i, \Lambda_i)$
 by the $H$--independent $(M_i, L_i')$ and  $ \mathfrak{v}_j$
 by an $H$--independent number.

\begin{defn}
 For any $ j= \{ i, \sigma (i) \} $ define $\mathfrak{vi}_j$ by

\begin{equation}\labelpar{eq:alphauj}
  \mathfrak{vi}_j := \left\{ \begin{array}{ccc}
                       \lambda_i  + \lambda_{ \sigma (i)} +  \mathfrak{v}_j  & \mbox{ if } & i \neq \sigma (i) \\
                     \lambda_i  + \mathfrak{v}_j & \mbox{ if } & i = \sigma (i) \\
                     \end{array} \right.
                     \mbox{.}
 \end{equation}
We call $ \mathfrak{vi}_j $ the \emph{vertical index} of $ \Sigma_j $.
\end{defn}

The next statement follows
from \ref{lem:bases} and \ref{th:mainth}  (see also  \ref{ss:gluinguj}).

\begin{cor}
For $ i \neq \sigma (i) $
\[
 \phi_{j*} ([M_i]) = - \phi_{j*} ([M_{\sigma(i)}]) \ \  \mbox{ and }  \]
\[ \phi_{j*} ([L_i']) = \phi_{j*} ([L_{ \sigma_i}']) + \mathfrak{vi}_j \cdot [ M_{ \sigma(i)}] \mbox{,}
\]
and for $ i = \sigma (i) $
\[ \phi_{j*} ([M_i] ) = -m \ \ \mbox{ and } \]
\[ \phi_{j*} ([L_i']) = c + \mathfrak{vi}_j \cdot m
\]
hold in the sense described in \ref{th:mainth}.
\end{cor}

\begin{cor} The integer
$ \mathfrak{vi}_j $ does not depend on the choice of $H$, thus it is an invariant of $f$ and $ \Sigma_j$.
\end{cor}

\subsection{The plumbing graph of $ \partial F $} We construct a plumbing graph for
$ \partial F$ by modifying a good embedded resolution graph of
$ (D,0)\subset (\C^2,0) $. (For  notations see subsection \ref{ss:embres}.)
The gluing of the plumbing construction uses  different set of longitudes
(cf. \cite{neumann1}). First we define them and then we rewrite the
above established identities regarding $\phi_{j*}$ in this language.

We choose small tubular neighborhoods $ N(E_v)$ of $E_v\subset\widetilde{\C^2}$ such that
\[  \pi^{-1} (\mathfrak{S}^3)\simeq
\partial \Big( \bigcup_{ v \in V} N(E_v) \Big) \ \ (\mbox{diffeomorphic
with $\mathfrak{S}^3$ via $\pi$)}\]
after smoothing the corners of $\bigcup_{ v \in V} N(E_v) $.
The tubular neighborhood $ N(E_v)$ of $ E_v \subset \widetilde{ \C^2 } $ is a $D^2$ (real 2--disk) bundle over $ E_v$ with Euler number $ e_v $.
We can choose this bundle structure in such a way
 that $ \tilde{D}_i$ is one of the  fibres of
$ N(E_{v(i)}) $ for each $ i= 1, \dots , l$.

We choose another generic fibre $\mathcal{F}_i \simeq D^2 $ of the bundle $ N(E_{v(i)}) $ near $ \tilde{D}_i $, and we set
$ L^{\pi}_i = \pi (\partial \mathcal{F}_i) \subset \mathfrak{S}^3$.
By the choice of the bundle structure of $ N(E_{v(i)}) $ and by the choice of the fibre $ \mathcal{F}_i$ we can assume that $ L^{\pi}_i \subset \partial N_i $.
\begin{defn} $L_i^\pi \subset \partial N_i\subset \mathfrak{S}^3 $ is called the
 \emph{resolution longitude}  of
$ L_i\subset \mathfrak{S}^3 $ associated with  the resolution $\pi$.
\end{defn}

We fix a resolution longitude $ L^{\pi}_i $ for each $ L_i $.
Clearly,
\begin{equation}
  H_1 (N_i, \Z )  = \Z \langle [ L^{\pi}_i ] \rangle \ \ \mbox{and} \ \
   H_1 ( \partial N_i, \Z ) =
 \Z \langle [ L^{\pi}_i ] \rangle \oplus \Z \langle [ M_i ] \rangle.
% where $ [ \gamma ] $ denotes the homology class of a closed curve $ \gamma $.
\end{equation}
The following facts are well--known (cf. \cite{EN}).
\begin{prop}
(a) $ \lk (L^{ \pi}_i , L_i ) = m_i(v(i)) $.

 (b) $  [ L^{ \pi}_i ] = [L'_i] + m_i(v(i)) \cdot [M_i]$ holds in $ H_1 ( \partial N_i, \Z )$.
\end{prop}

%\begin{proof}
%  Part {\em (b)} follows from part {\em  (a)}. For statement (a) we consider the %Milnor fibre $ F_{d_i}= d_i^{-1} ( \delta ) \cap B^4 $ and its boundary
%$ \partial F_{d_i} = d_i^{-1} ( \delta ) \cap S^3 $. Then
 %\[
 % \lk (L^{ \pi}_i , L_i ) =  \lk (L^{ \pi}_i , \partial F_{d_i} ) \mbox{.}
% \]
%$ \lk (L^{ \pi}_i , \partial F_{d_i} ) ) $ is equal to the algebraic number
%of the intersection points of $ {\pi}(\mathcal{F}_i) $ and $ F_{d_i} $.
%That is the same to take the algebraic number of intersection points of
%$ \mathcal{F}_i $ and ${\pi}^{-1} ( F_{d_i} ) $ in $ \widetilde{ \C^2 }$,
%because $ 0 \notin F_{d_i} $. $ \mathcal{F}_i $ and ${\pi}^{-1} ( F_{d_i} ) $
%intersect each other transversally in $ m_i (v(i)) $ points, each of them
%has positive sign.
%
%\end{proof}

\begin{cor}\label{cor:REL}
 $  [ L^{ \pi}_i ] = [\Lambda_i] + ( m_i(v(i)) - \lambda_i ) \cdot [M_i]$ holds in $ H_1 ( \partial N_i, \Z )$.
\end{cor}

\begin{defn}\label{de:alpha}
For any $ j=\{i,\sigma(i)\}$ define $ \alpha_j $ by:
 \begin{equation}\labelpar{eq:alpha}
  \alpha_j = \left\{ \begin{array}{ccc}
                      - m_i(v(i)) + \lambda_i -m_{ \sigma (i)} ( v( \sigma (i))) + \lambda_{ \sigma (i)} +  \mathfrak{v}_j  & \mbox{ if } & i \neq \sigma (i) \\
                     \lambda_i -  m_i(v(i)) + \mathfrak{v}_j & \mbox{ if } & i = \sigma (i). \\
                     \end{array} \right.
 \end{equation}
\end{defn}
% Note that $ \alpha_i = \alpha_{ \sigma (i)} $.
 Then Theorem \ref{th:mainth} and Corollary \ref{cor:REL} give the following.

\begin{cor}\label{cor:identities} {\bf Case 1:} \
 For $ i \neq \sigma (i) $ in $ H_1 (S^1\times S^1, \Z )$
 the following identities  hold:
  \[ \phi_{j*} ([M_i]) = - \phi_{j*} ([M_{\sigma(i)}])\ \ \mbox{ and } \ \
  \phi_{j*} ([L^{ \pi}_i ]) =
 \phi_{j * } ( [L^{ \pi}_{ \sigma (i)}] + \alpha_j  \cdot [M_{ \sigma (i) }]).
 \]
 {\bf Case 2:} \
  For $ i= \sigma (i) $ in $ H_1 (\partial Y, \Z) $ the following
  identities hold:
 \[ \phi_{j*} ([M_i] ) = -m  \ \ \mbox{and} \ \
 \phi_{j *} ( [L^{ \pi }_i] ) = c + \alpha_j \cdot m.
 \]
\end{cor}

\subsection{The construction of the plumbing graph.} From the
embedded resolution graph
$ \Gamma $ of $ (D, 0) \subset ( \C^2, 0) $ associated with
the resolution ${\pi}$ we construct a plumbing graph $\widehat{\Gamma }$.

Recall that $\Gamma$ has $l$ arrowhead  vertices representing the strict transforms
of $\{D_i\}_{i=1}^l$, and the $i$-th arrowhead is supported (via a unique edge) by
the vertex $v(i)\in V$. We obtain the plumbing graph $\widehat{\Gamma}$ from $\Gamma$ as follows.

(1) Fix $j=\{i,\sigma(i)\}$ and assume that $i\not=\sigma(i)$. Then we identify the two
arrowheads and we replace it by a single new vertex of $\widehat{\Gamma}$.
 We define the Euler number of the new  vertex by
$\alpha_j$. Both two edges (which in $\Gamma$ supported  the arrowheads) will survive as edges of this new vertex
(connecting it with $v(i)$ and $v(\sigma(i))$ respectively),
however one of them will have a negative sign, the other one a
positive sign.  By plumbing calculus, cf. \cite[Prop. 2.1. R0(a)]{neumann1},
the choice of the edge which has negative sign -- denoted by  $\ominus$ --  is irrelevant. (The edges without any decorations, by convention, are edges
 with positive sign.)

(2) Fix $j=\{i,\sigma(i)\}$ such that $i=\sigma(i)$. Then the arrowhead associated with
$i=\sigma(i)$ will be replaced by a new vertex and it will be decorated by Euler number $\alpha_j$.
Furthermore, to this new vertex
we attach the plumbing graph of $Y$ (cf. \ref{ss:plY}) as indicated below.
The  edge connecting the
new vertex and the graph of $Y$  will have a negative sign $\ominus$.

(3) We do all these modification for all $j\in J$, otherwise we keep the shape and decorations of $\Gamma$.

\vspace{2mm}

More precisely,
if the schematic picture of $\Gamma$ is the following,

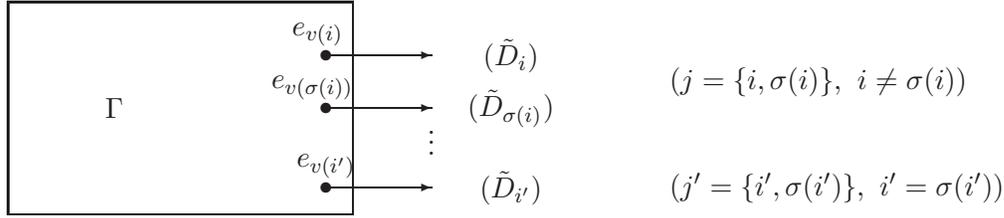
\begin{picture}(300,100)(-20,0)
\put(0,10){\framebox(130,80)}
\put(120,20){\vector(1,0){40}}
\put(120,28){\makebox(0,0){$e_{v(i')} $}}
\put(120,20){\circle*{4}}
\put(160,40){\makebox(0,0){$ \vdots $}}
\put(190,20){\makebox(0,0){$(\tilde{D}_{i'})$}}
\put(40,50){\makebox(0,0){$\Gamma$}}

\put(120,50){\vector(1,0){40}}
\put(115,58){\makebox(0,0){$e_{v(\sigma(i))} $}}
\put(120,50){\circle*{4}}
\put(190,50){\makebox(0,0){$(\tilde{D}_{\sigma(i)})$}}
\put(120,70){\vector(1,0){40}}
\put(117,78){\makebox(0,0){$e_{v(i)} $}}
\put(120,70){\circle*{4}}
\put(190,70){\makebox(0,0){$(\tilde{D}_i)$}}

\put(250,20){\makebox(0,0)[l]{$(j'=\{i',\sigma(i')\}, \ i'=\sigma(i'))$}}

\put(250,60){\makebox(0,0)[l]{$(j=\{i,\sigma(i)\}, \ i\not=\sigma(i))$}}

\end{picture}

 \noindent then the schematic picture of $\widehat{\Gamma}$ is

 
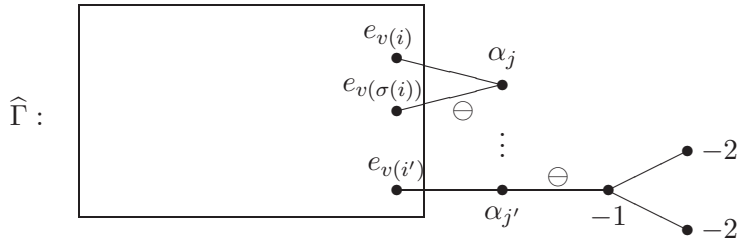
\begin{picture}(300,105)(-20,-5)
\put(0,10){\framebox(130,80)}
\put(120,20){\line(1,0){80}}
\put(120,28){\makebox(0,0){$e_{v(i')} $}}
\put(120,20){\circle*{4}}

\put(160,20){\circle*{4}}
\put(200,20){\circle*{4}}
\put(230,35){\circle*{4}}
\put(230,5){\circle*{4}}
\put(200,20){\line(2,1){30}}
%\put(60,60){\line(1,0){40}}
\put(200,20){\line(2,-1){30}}
\put(200,10){\makebox(0,0){$-1$}}
\put(242,5){\makebox(0,0){$-2$}}
\put(242,35){\makebox(0,0){$-2$}}
\put(160,10){\makebox(0,0){$ \alpha_{j'}$}}
\put(181,25){\makebox(0,0){$\ominus$}}
\put(160,40){\makebox(0,0){$ \vdots $}}

\put(120,50){\line(4,1){40}}
\put(115,58){\makebox(0,0){$e_{v(\sigma(i))} $}}
\put(120,50){\circle*{4}}\put(160,60){\circle*{4}}
\put(160,70){\makebox(0,0){$\alpha_j$}}
\put(120,70){\line(4,-1){40}}
\put(117,78){\makebox(0,0){$e_{v(i)} $}}
\put(120,70){\circle*{4}}

\put(-20,50){\makebox(0,0){$\widehat{\Gamma}:$}}

\put(145,50){\makebox(0,0){$\ominus$}}

\end{picture}

In fact, by plumbing calculus (cf. \cite[Prop. 2.1. R0(a)]{neumann1}),  whenever $i=\sigma(i)$
the edge sign from the newly created `branch' (subtree) can be omitted, however at this point we put it since this is
the graph provided by the proof (which reflects properly the corresponding base changes).

Note also that usually the graph $\widehat{\Gamma}$ (that is, the associated intersection form) is not negative definite (or,  it is not
even equivalent via plumbing calculus by a negative definite graph).

Above (when $i\not=\sigma(i)$) the coincidence $v(i)=v(\sigma(i))$ might happen, in fact, in all the cases we
analyzed this coincidence (on minimal graph) does happen.

\begin{thm}
 The plumbing $3$-manifold associated with the plumbing graph
 $ \widehat{\Gamma} $ is orientation preserving diffeomorphic with $ \partial F $.
\end{thm}
\begin{proof}
This follows from Corollary \ref{cor:identities} and the relationships between this base--change
and the plumbing construction as it is described  in \cite{neumann1}, pages 318--319.
In the case $i=\sigma(i)$ the plumbing graph of $Y$ from \ref{ss:plY} should be also used.
Note that in both cases Corollary \ref{cor:identities} provides a base change matrix
from the left hand side of (\ref{eq:matrix}), this decomposes as a product as in the right hand side.
\begin{equation}\label{eq:matrix}
\begin{pmatrix}-1 & \alpha_j\\ 0& 1\end{pmatrix}=
\begin{pmatrix}0 & 1\\ 1& 0\end{pmatrix}
\begin{pmatrix}-1 & 0\\ -\alpha_j& 1\end{pmatrix}
\begin{pmatrix} 0 & -1 \\ -1& 0\end{pmatrix}.
\end{equation}
This, according to \cite[pages 318--320]{neumann1} is interpreted as a `gluing' by a string (with length one),
this is the new vertex (for each $j$) given by the construction of $\widehat{\Gamma}$.
\end{proof}

%$ \widetilde{ B^4} := p^{-1} (B^4) \subset \widetilde{ \C^2} $ is the union of tubular neighborhoods of the exceptional divisors $ E_v \subset \widetilde{ \C^2 } $.
%$ \partial \widetilde{ B^4} \approx S^3 $ can be optained by plumbing construction along the graph $ G$.

% Each vertex $ v \in V $ marks an $ S^1 $-bundle over an exceptional divisor $\C P^1 \approx S^2 $ with Euler number $ e(v) \in \Z $. The arrows $a_1, \dots , a_l \in A $ correspond to the components $ L_i $ of the link $ \tilde{D}$. Each vertex $ v \in V $ is endowed with the multiplicities $ m_1(v), \dots, m_l(v) \in \N $ corresponding to the components $ \tilde{ \Sigma }_i $ od $ \tilde{ \Sigma } $. The edge $ e \in E $ connecting the vertexes $ v $ and $ v' $ marks the gluing of the total spaces of the bundles corresponding to $ v $ and $ v' $ along the boundaries of cutted solid tori. Each edge is endowed with a sign $ \oplus $ or $ \ominus $ which describes the type of the gluing.

% For an arrow $ a_i \in V $ let $ v(a_i) \in V $ denote the vertex where the arrow starts. $ L_i $ is an $ S^1$-fibre of the bundle corresponding to $ v(a_i) $.

% if $ v=v(i) $, and
% \[
% e_v \cdot m_i(v(i)) + \Sigma_{w \in V} (E_{v(i)} \cdot E_w) \cdot m_i(w) = 0 \mbox{.}\]

\section{The vertical index for $\Sigma^{1,0}$ type germs}

\subsection{$\Sigma^{1,0}$ type germs} Assume that $\Phi(s, t) = ( s, t^2, t d(s,t) ) $, where $d(s,t) = g(s, t^2) $ for  some germ $g$, such that $d(s,t)$  is not divisible by $t$. In this case the equation of the image is $f= yg^2(x,y) - z^2 =0 $. These germs (more precisely, their $ \mathcal{A}$-equivalence classes) are labeled by the Boardman symbol $ \Sigma^{1,0}$, and they are exactly the corank--$1$ map germs with no triple points in their stabilization,
see \cite{Mond0,nunodoodle} for details.

 The set of double points is $D = \{ (s, t) \in \C^2 \ | \  d(s,t)= 0 \}$. It is
 equipped with the involution $ \iota: D \to D $, $ \iota (s, t) = (s, -t) $.
 The set of double values is
 \[ \Sigma (f) = \{ (x, y, z) \in \C^3 \ | \  z=0 \mbox{ and } g(x,y) =0 \} \mbox{.}\]

There are several options for the choice of the  transversal section. E.g.,
Proposition \ref{pr:H} works with $a_1=a_2=a_3=1$. Furthermore,  one can also take
$H_2 (x, y, z) = z $ or $  H_3 (x, y, z) = g(x, y)$.

If we take $H(x,y,z)=z$ then the $H$--vertical indexes
$ \mathfrak{v}_j$ are zero. This fact  follows directly from the product decomposition of $f$.
(However, for different other choices it can be nonzero as well.)

\begin{prop} For $ \Sigma^{1,0} $--type germs
$ \Sigma_j \mathfrak{vi}_j = -  \Sigma_{i \neq k} D_i \cdot D_k - C( \Phi ) $, where $ C(\Phi) $ is the number of crosscaps of a stabilization of $ \Phi$, cf. \cite{NP, Mond2}.
\end{prop}

\begin{proof}
 With the choice of $ H=z$, $ \Phi^*(H)= td(s,t) $ and $ \mathfrak{v}_j=0$. If $i \neq \sigma(i)$, then
\[ \mathfrak{vi}_j= -\Sigma_{k \not \in \{i,\sigma(i)\}}
( D_i + D_{\sigma(i)}) \cdot D_k - (D_i+D_{\sigma(i)}) \cdot \{t=0\}, \]
while for  $i= \sigma(i) $  (see \ref{ss:topverindex}) one has
\[ \mathfrak{vi}_j= -\Sigma_{k \neq i} D_i \cdot D_k - D_i \cdot \{t=0\}.
 \]
Taking the sum for all $j$ we get the identity, once we verify
 \[ \Sigma_{i=1}^l D_i \cdot \{t=0 \} = D \cdot \{t=0\} =
  \dim_{ \C} \frac{ \mathcal{O}_{ ( \C^2, 0)}}{ \langle t, d(s,t) \rangle} \mbox{.}
 \]
 Here, the last identity follows from the algebraic definition of $C( \Phi)$, as the codimension of the Jacobian ideal of $ \Phi $ (cf. \cite{Mond2, NP, nunodouble}). In our case this ideal is genereted by $ t$ and $d(s, t) $.
 \end{proof}

\section{Examples}

In the next paragraphs we provide some concrete examples. The families are taken from
D. Mond's list  of simple germs  \cite[Table 1]{Mond2}. In the sequel
we provide the resolution graph of $D$ and the plumbing graph of $\partial F$. The  computations are left to the reader.

The last two examples are more special than the previous ones: they are not of type
$\Sigma^{1,0}$. In the first family ($ H_k $ from the list \cite[Table 1]{Mond2})
the calculation  has some nontrivial steps, hence
 we provide more details. The last example in \ref{ss:2cor} is a corank
 $2$ map germ from \cite{marar}.

\subsection{Whitney umbrella, or cross--cap} $\Phi (s, t) = (s, t^2, ts) $ and
$ f(x, y, z) = x^2y - z^2 =0$.
 The graph of $D$ is on the left, while
 our algorithm provides the  graph from the right for
$ \partial F$

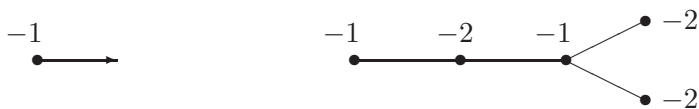
\begin{picture}(300,70)(-160,25)
 \put(-100,60){\circle*{4}}
                         \put(-100,60){\vector(1,0){30}}
                         \put(-105,70){\makebox(0,0){$-1$}}

                        \put(100,60){\circle*{4}}
                           \put(20,60){\circle*{4}}
                        \put(60,60){\circle*{4}}
                         \put(130,75){\circle*{4}}
                         \put(130,45){\circle*{4}}
                         \put(100,60){\line(2,1){30}}
                         \put(20,60){\line(1,0){40}}
                         \put(60,60){\line(1,0){40}}
                         \put(100,60){\line(2,-1){30}}
                         \put(95,70){\makebox(0,0){$-1$}}
                          \put(15,70){\makebox(0,0){$-1$}}
                          \put(143,75){\makebox(0,0){$-2$}}
                          \put(58,70){\makebox(0,0){$ -2 $}}
                           \put(143,45){\makebox(0,0){$-2$}}
                           \end{picture}

\noindent which after plumbing calculus transforms into
                           \begin{picture}(20,10)(50,58)
                        \put(60,60){\circle*{4}}
                          \put(58,70){\makebox(0,0){$ -4 $}}
                           \end{picture}
(Compare also with Example 10.4.2 from \cite{NSz}.)

\subsection{\bf{$S_1$}}
%\marginpar{mathbf}
Set
$ \Phi(s, t) = ( s, t^2, t^3 + s^2t ) $, hence
$f(x, y, z) = y(y+x^2)^2 - z^2$.
The graph of $D$ is the first diagram, while the other two equivalent
graphs represent $\partial F$

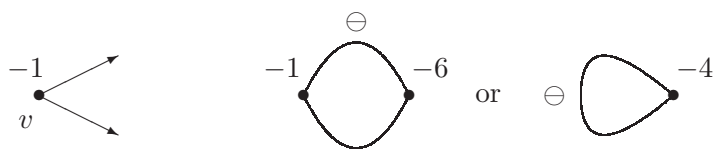
\begin{picture}(300,60)(-50,40)
 \put(0,60){\circle*{4}}
                         \put(0,60){\vector(2,1){30}}
                         \put(0,60){\vector(2,-1){30}}
                         \put(-5,70){\makebox(0,0){$-1$}}
                          \put(-5,50){\makebox(0,0){$v$}}

                         \put(100,60){\circle*{4}}
                       \put(140,60){\circle*{4}}
                       % \put(120,60){\circle{40}}
 \qbezier(100,60)(120,100)(140,60)\qbezier(100,60)(120,20)(140,60)
                         \put(92,70){\makebox(0,0){$-1$}}
                             \put(148,70){\makebox(0,0){$-6$}}
                              \put(120,88){\makebox(0,0){$ \ominus $}}
                            %  \put(120,32){\makebox(0,0){$ \oplus $}}                 %\end{picture}
 \put(170,60){\makebox(0,0){$ \mbox{or} $}}
%\begin{picture}(300,100)(-50,10)
                       \put(240,60){\circle*{4}}
                        %  \put(220,60){\circle{40}}
  \qbezier(240,60)(205,90)(205,60)  \qbezier(240,60)(205,30)(205,60)
                             \put(248,70){\makebox(0,0){$-4$}}
                              \put(195,60){\makebox(0,0){$ \ominus $}}
                         \end{picture}

\subsection{The family $S_{k-1} $, $k\geq 2$}
%For $k=1$ by $S_0$ one recovers the  gross--cap.
 One has $ \Phi(s, t) = ( s, t^2, t^3 + s^k t )$ and
$ f(x, y, z) = y(y+x^k)^2 - z^2 =0 $.

 \vspace{2mm}

 \emph{Case 1}: $k=2n$.

 
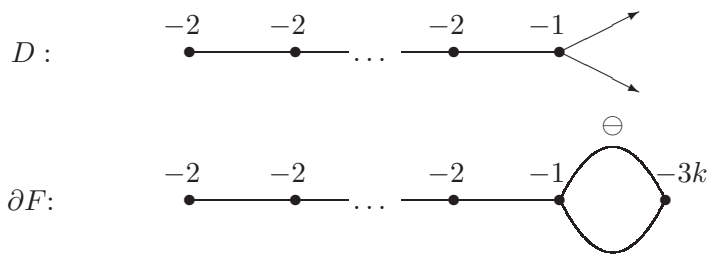
\vspace{2mm}

\begin{picture}(300,40)(-150,45)
                         \put(100,60){\circle*{4}}
                          \put(60,60){\circle*{4}}
                           \put(0,60){\circle*{4}}
                           \put(60,60){\line(1,0){40}}
                             \put(40,60){\line(1,0){20}}
                             \put(0,60){\line(1,0){20}}
                              \put(-40,60){\circle*{4}}
                              \put(-40,60){\line(1,0){40}}
                         \put(100,60){\vector(2,1){30}}
                         \put(100,60){\vector(2,-1){30}}
                         \put(95,70){\makebox(0,0){$-1$}}
                          \put(57,70){\makebox(0,0){$-2$}}
                           \put(-3,70){\makebox(0,0){$-2$}}
                            \put(-43,70){\makebox(0,0){$-2$}}
                         % \put(93,50){\makebox(0,0){$w_n$}}
                         % \put(-43,50){\makebox(0,0){$w_1$}}
                         % \put(-3,50){\makebox(0,0){$w_2$}}
                          \put(29,57){\makebox(0,0){$\dots$}}
                         % \put(60,50){\makebox(0,0){$w_{n-1}$}}
                           \put(-100,60){\makebox(0,0){$D:$}}
                         \end{picture}

\begin{picture}(300,60)(-150,40)
                         \put(100,60){\circle*{4}}
                          \put(60,60){\circle*{4}}
                           \put(0,60){\circle*{4}}
                           \put(140,60){\circle*{4}}
                       %    \put(120,60){\circle{40}}
\qbezier(100,60)(120,100)(140,60)\qbezier(100,60)(120,20)(140,60)
                           \put(60,60){\line(1,0){40}}
                             \put(40,60){\line(1,0){20}}
                             \put(0,60){\line(1,0){20}}
                              \put(-40,60){\circle*{4}}
                              \put(-40,60){\line(1,0){40}}
                         \put(95,70){\makebox(0,0){$-1$}}
                          \put(57,70){\makebox(0,0){$-2$}}
                           \put(-3,70){\makebox(0,0){$-2$}}
                            \put(-43,70){\makebox(0,0){$-2$}}
                             \put(146,70){\makebox(0,0){$-3k$}}
                               \put(120,88){\makebox(0,0){$\ominus$}}
                             \put(29,57){\makebox(0,0){$\dots$}}
\put(-100,60){\makebox(0,0){$\partial F$:}}
                         \end{picture}

\noindent Here the number of $ (-2)$-vertices is $n-1$.

 \vspace{2mm}

 \emph{Case 2}: $k=2n+1$.

 
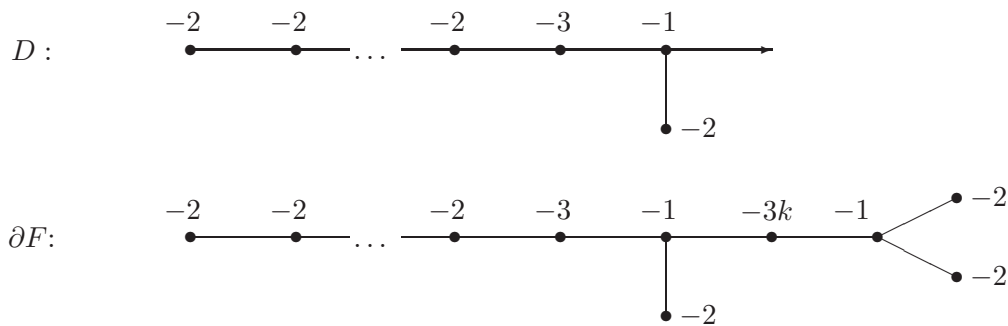
\vspace{2mm}

\begin{picture}(300,50)(-100,40)
 \put(140,60){\circle*{4}}
  \put(-100,60){\makebox(0,0){$D:$}}
                         \put(100,60){\circle*{4}}
                          \put(60,60){\circle*{4}}
                           \put(0,60){\circle*{4}}
                            \put(140,30){\circle*{4}}
                           \put(60,60){\line(1,0){40}}
                             \put(40,60){\line(1,0){20}}
                             \put(0,60){\line(1,0){20}}
                              \put(-40,60){\circle*{4}}
                              \put(-40,60){\line(1,0){40}}
                              \put(100,60){\line(1,0){40}}
                              \put(140,60){\line(0,-1){30}}
                         \put(140,60){\vector(1,0){40}}
                         \put(137,70){\makebox(0,0){$-1$}}
                         \put(97,70){\makebox(0,0){$-3$}}
                          \put(57,70){\makebox(0,0){$-2$}}
                           \put(-3,70){\makebox(0,0){$-2$}}
                            \put(-43,70){\makebox(0,0){$-2$}}
                        %  \put(100,50){\makebox(0,0){$w_{n}$}}
                        %  \put(-43,50){\makebox(0,0){$w_1$}}
                         % \put(-3,50){\makebox(0,0){$w_2$}}
                          \put(29,57){\makebox(0,0){$\dots$}}
                         % \put(60,50){\makebox(0,0){$w_{n-1}$}}
                         % \put(132,50){\makebox(0,0){$v_2$}}
                         %   \put(132,30){\makebox(0,0){$v_1$}}
                  \put(152,30){\makebox(0,0){$-2$}}
                         \end{picture}

\begin{picture}(300,100)(-100,10)
\put(-100,60){\makebox(0,0){$\partial F$:}}
\put(140,60){\circle*{4}}
                         \put(100,60){\circle*{4}}
                          \put(60,60){\circle*{4}}
                           \put(0,60){\circle*{4}}
                            \put(140,30){\circle*{4}}
                           \put(60,60){\line(1,0){40}}
                             \put(40,60){\line(1,0){20}}
                             \put(0,60){\line(1,0){20}}
                              \put(-40,60){\circle*{4}}
                              \put(-40,60){\line(1,0){40}}
                              \put(100,60){\line(1,0){40}}
                              \put(140,60){\line(0,-1){30}}
                         \put(140,60){\line(1,0){40}}
                         \put(137,70){\makebox(0,0){$-1$}}
                         \put(97,70){\makebox(0,0){$-3$}}
                          \put(57,70){\makebox(0,0){$-2$}}
                           \put(-3,70){\makebox(0,0){$-2$}}
                            \put(-43,70){\makebox(0,0){$-2$}}
                          \put(29,57){\makebox(0,0){$\dots$}}
                  \put(152,30){\makebox(0,0){$-2$}}
                  \put(220,60){\circle*{4}}
                        \put(180,60){\circle*{4}}
                         \put(250,75){\circle*{4}}
                         \put(250,45){\circle*{4}}
                         \put(220,60){\line(2,1){30}}
                         \put(180,60){\line(1,0){40}}
                         \put(220,60){\line(2,-1){30}}
                         \put(210,70){\makebox(0,0){$-1$}}
                          \put(262,75){\makebox(0,0){$-2$}}
                          \put(178,70){\makebox(0,0){$ -3k $}}
                           \put(262,45){\makebox(0,0){$-2$}}
                         \end{picture}

\noindent where the number of $(-2)$-vertices on the left is $n-1$.

\subsection{The family $B_{k} $ ($k\geq 1$)}
$ \Phi(s, t) = ( s, t^2, s^2t + t^{2k+1} )$ and
$ f= y(x^2+y^k)^2 - z^2 =0 $.

The graph of $D$ is

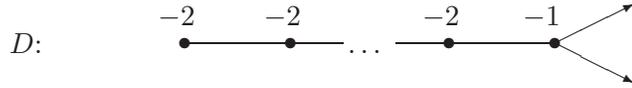
\begin{picture}(300,60)(-150,30)
\put(-100,60){\makebox(0,0){$ D $:}}
                         \put(100,60){\circle*{4}}
                          \put(60,60){\circle*{4}}
                           \put(0,60){\circle*{4}}
                           \put(60,60){\line(1,0){40}}
                             \put(40,60){\line(1,0){20}}
                             \put(0,60){\line(1,0){20}}
                              \put(-40,60){\circle*{4}}
                              \put(-40,60){\line(1,0){40}}
                         \put(100,60){\vector(2,1){30}}
                         \put(100,60){\vector(2,-1){30}}
                         \put(95,70){\makebox(0,0){$-1$}}
                          \put(57,70){\makebox(0,0){$-2$}}
                           \put(-3,70){\makebox(0,0){$-2$}}
                            \put(-43,70){\makebox(0,0){$-2$}}
                          \put(29,57){\makebox(0,0){$\dots$}}
                         % \put(60,50){\makebox(0,0){$w_{k-1}$}}
                         \end{picture}

\noindent  where the number of $( -2)$-vertices is $k-1$.

Note that the double point curve $ D $ does not depend on the parity of $k$,
however $ \Sigma $ has one component, when $ k$ is odd, and one component,
when $k$ is even. Thus the pairing $ \sigma $ changes the components of
$ D$ in the odd case, and $ \sigma $ is the identity in the even case.

\color{black}

 \vspace{2mm}

 \emph{Case 1}: $k=2n+1$.

 \vspace{2mm}

%\begin{picture}(300,50)(-150,45)
 % \put(-100,60){\makebox(0,0){$D:$}}
 %                        \put(100,60){\circle*{4}}
 %                         \put(60,60){\circle*{4}}
 %                          \put(0,60){\circle*{4}}
 %                          \put(60,60){\line(1,0){40}}
 %                            \put(40,60){\line(1,0){20}}
 %                            \put(0,60){\line(1,0){20}}
 %                             \put(-40,60){\circle*{4}}
 %                             \put(-40,60){\line(1,0){40}}
 %                        \put(100,60){\vector(2,1){30}}
 %                        \put(100,60){\vector(2,-1){30}}
 %                        \put(95,70){\makebox(0,0){$-1$}}
 %                         \put(57,70){\makebox(0,0){$-2$}}
 %                          \put(-3,70){\makebox(0,0){$-2$}}
 %                           \put(-43,70){\makebox(0,0){$-2$}}
 %                      %   \put(93,50){\makebox(0,0){$w_k$}}
 %                       %  \put(-43,50){\makebox(0,0){$w_1$}}
 %                       %  \put(-3,50){\makebox(0,0){$w_2$}}
 %                         \put(29,57){\makebox(0,0){$\dots$}}
 %                       %  \put(60,50){\makebox(0,0){$w_{k-1}$}}
 %                        \end{picture}

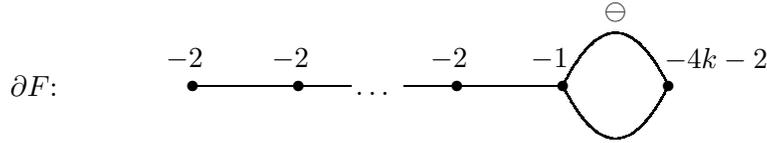
\begin{picture}(300,65)(-150,30)
\put(-100,60){\makebox(0,0){$\partial F$:}}
                         \put(100,60){\circle*{4}}
                          \put(60,60){\circle*{4}}
                           \put(0,60){\circle*{4}}
                           \put(140,60){\circle*{4}}
                          % \put(120,60){\circle{40}}
\qbezier(100,60)(120,100)(140,60)\qbezier(100,60)(120,20)(140,60)
                           \put(60,60){\line(1,0){40}}
                             \put(40,60){\line(1,0){20}}
                             \put(0,60){\line(1,0){20}}
                              \put(-40,60){\circle*{4}}
                              \put(-40,60){\line(1,0){40}}
                         \put(95,70){\makebox(0,0){$-1$}}
                          \put(57,70){\makebox(0,0){$-2$}}
                           \put(-3,70){\makebox(0,0){$-2$}}
                            \put(-43,70){\makebox(0,0){$-2$}}
                             \put(158,70){\makebox(0,0){$-4k-2$}}
                               \put(120,88){\makebox(0,0){$\ominus$}}
                             \put(29,57){\makebox(0,0){$\dots$}}
                         \end{picture}

\noindent
where the number of $( -2)$-vertices is $k-1$.

 \vspace{2mm}

 \emph{Case 2}: $k=2n$.

 \vspace{2mm}

%\begin{picture}(300,100)(-150,10)
%                         \put(100,60){\circle*{4}}
%                          \put(60,60){\circle*{4}}
%                           \put(0,60){\circle*{4}}
%                           \put(60,60){\line(1,0){40}}
%                             \put(40,60){\line(1,0){20}}
%                             \put(0,60){\line(1,0){20}}
%                              \put(-40,60){\circle*{4}}
%                              \put(-40,60){\line(1,0){40}}
%                         \put(100,60){\vector(1,1){30}}
%                         \put(100,60){\vector(1,-1){30}}
%                         \put(95,70){\makebox(0,0){$-1$}}
%                          \put(57,70){\makebox(0,0){$-2$}}
%                           \put(-3,70){\makebox(0,0){$-2$}}
%                            \put(-43,70){\makebox(0,0){$-2$}}
%                          \put(93,50){\makebox(0,0){$w_k$}}
%                          \put(-43,50){\makebox(0,0){$w_1$}}
%                          \put(-3,50){\makebox(0,0){$w_2$}}
%                          \put(29,57){\makebox(0,0){$\dots$}}
%                          \put(60,50){\makebox(0,0){$w_{k-1}$}}
%                         \end{picture}

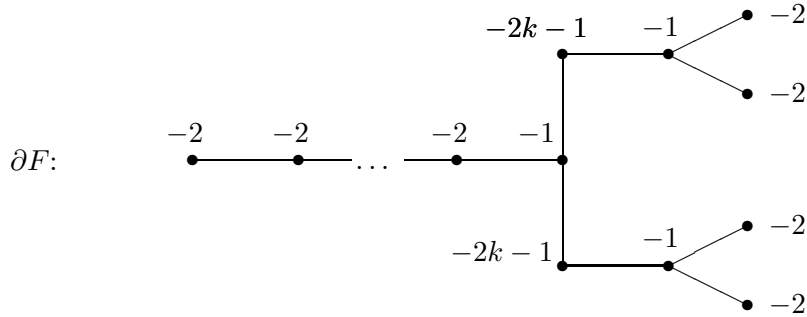
\begin{picture}(300,140)(-150,-20)
\put(-100,60){\makebox(0,0){$\partial F$:}}
                         \put(100,60){\circle*{4}}
                          \put(60,60){\circle*{4}}
                           \put(0,60){\circle*{4}}
                           \put(60,60){\line(1,0){40}}
                             \put(40,60){\line(1,0){20}}
                             \put(0,60){\line(1,0){20}}
                              \put(-40,60){\circle*{4}}
                              \put(-40,60){\line(1,0){40}}
                         \put(100,60){\line(0,1){40}}
                         \put(100,60){\line(0,-1){40}}
                         \put(100,100){\circle*{4}}
                         \put(100,20){\circle*{4}}
                         \put(90,70){\makebox(0,0){$-1$}}
                          \put(57,70){\makebox(0,0){$-2$}}
                           \put(-3,70){\makebox(0,0){$-2$}}
                            \put(-43,70){\makebox(0,0){$-2$}}
                          \put(29,57){\makebox(0,0){$\dots$}}
                             \put(140,100){\circle*{4}}
                         \put(140,20){\circle*{4}}
                          \put(100,100){\line(1,0){40}}
                           \put(100,20){\line(1,0){40}}
                            \put(140,100){\line(2,1){30}}
                               \put(140,100){\line(2,-1){30}}
                                  \put(140,20){\line(2,1){30}}
                                     \put(140,20){\line(2,-1){30}}
                                      \put(170,115){\circle*{4}}
                                       \put(170,85){\circle*{4}}
                                        \put(170,35){\circle*{4}}
                                         \put(170,5){\circle*{4}}
                             \put(90,110){\makebox(0,0){$-2k-1$}}
                              \put(77,25){\makebox(0,0){$-2k-1$}}
                               \put(137,110){\makebox(0,0){$-1$}}
                              \put(137,30){\makebox(0,0){$-1$}}
                               \put(90,110){\makebox(0,0){$-2k-1$}}
                              \put(185, 115){\makebox(0,0){$-2$}}
                           \put(185, 85){\makebox(0,0){$-2$}}
                            \put(185, 35){\makebox(0,0){$-2$}}
                             \put(185,5){\makebox(0,0){$-2$}}
                         \end{picture}

\noindent
where the number of $(-2)$-vertexes on the left is again $k-1$.

\subsection{The family $C_k $ ($k\geq 1$)}
 $\Phi(s, t) = ( s, t^2, st^3 + s^k t )$ and
$ f = y(xy+x^k)^2 - z^2 =0$.

 \vspace{2mm}

 \emph{Case 1}: $k=2n+1$.

 
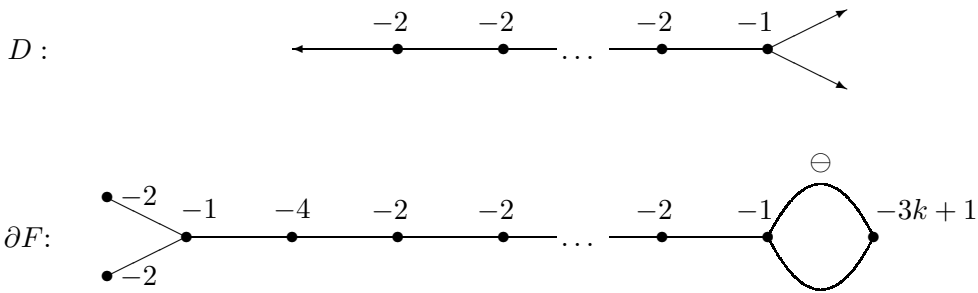
\vspace{2mm}

\begin{picture}(300,50)(-200,40)
 \put(-180,60){\makebox(0,0){$D:$}}
                         \put(100,60){\circle*{4}}
                          \put(60,60){\circle*{4}}
                           \put(0,60){\circle*{4}}
                           \put(60,60){\line(1,0){40}}
                             \put(40,60){\line(1,0){20}}
                             \put(0,60){\line(1,0){20}}
                              \put(-40,60){\circle*{4}}
                              \put(-40,60){\line(1,0){40}}
                         \put(100,60){\vector(2,1){30}}
                       \put(-40,60){\vector(-1,0){40}}
                         \put(100,60){\vector(2,-1){30}}
                         \put(95,70){\makebox(0,0){$-1$}}
                          \put(57,70){\makebox(0,0){$-2$}}
                           \put(-3,70){\makebox(0,0){$-2$}}
                            \put(-43,70){\makebox(0,0){$-2$}}
                       %   \put(93,50){\makebox(0,0){$w_n$}}
                       %   \put(-43,50){\makebox(0,0){$w_1$}}
                      %    \put(-3,50){\makebox(0,0){$w_2$}}
                          \put(29,57){\makebox(0,0){$\dots$}}
                     %     \put(60,50){\makebox(0,0){$w_{n-1}$}}
                    %       \put(-78,50){\makebox(0,0){$a_1$}}
                   % \put(140,90){\makebox(0,0){$a_2$}}
                  %   \put(140,30){\makebox(0,0){$a_3$}}
                         \end{picture}

\begin{picture}(300,70)(-200,40)
\put(-180,60){\makebox(0,0){$\partial F$:}}
                         \put(100,60){\circle*{4}}
                          \put(60,60){\circle*{4}}
                           \put(0,60){\circle*{4}}
                           \put(140,60){\circle*{4}}
                          % \put(120,60){\circle{40}}
\qbezier(100,60)(120,100)(140,60)\qbezier(100,60)(120,20)(140,60)
                           \put(60,60){\line(1,0){40}}
                             \put(40,60){\line(1,0){20}}
                             \put(0,60){\line(1,0){20}}
                              \put(-40,60){\circle*{4}}
                              \put(-40,60){\line(1,0){40}}
                         \put(95,70){\makebox(0,0){$-1$}}
                          \put(57,70){\makebox(0,0){$-2$}}
                           \put(-3,70){\makebox(0,0){$-2$}}
                            \put(-43,70){\makebox(0,0){$-2$}}
                             \put(160,70){\makebox(0,0){$-3k+1$}}
                               \put(120,88){\makebox(0,0){$\ominus$}}
                             \put(29,57){\makebox(0,0){$\dots$}}
                              \put(-40,60){\line(-1,0){40}}
                                \put(-120,60){\circle*{4}}
                                \put(-120,60){\line(-2,1){30}}
                                \put(-120,60){\line(-2,-1){30}}
                                 \put(-150,75){\circle*{4}}
                                 \put(-150,45){\circle*{4}}
                                  \put(-80,60){\circle*{4}}
                                   \put(-80,60){\line(-1,0){40}}
                                    \put(-138,75){\makebox(0,0){$-2$}}
                                          \put(-138,45){\makebox(0,0){$-2$}}
                                           \put(-80,70){\makebox(0,0){$-4$}}
                                            \put(-115,70){\makebox(0,0){$-1$}}
                         \end{picture}

\noindent where the number of $ (-2)$-vertices in the middle is $n-1$.

 \vspace{2mm}

 \emph{Case 2}: $k=2n$.

 
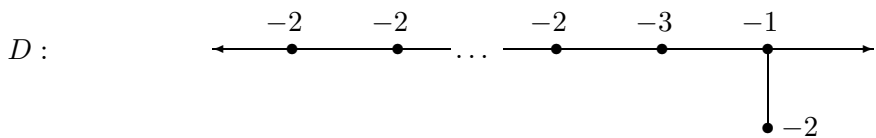
\vspace{2mm}

\begin{picture}(300,50)(-160,30)
 \put(-140,60){\makebox(0,0){$D:$}}
 \put(140,60){\circle*{4}}
                         \put(100,60){\circle*{4}}
                          \put(60,60){\circle*{4}}
                           \put(0,60){\circle*{4}}
                            \put(140,30){\circle*{4}}
                           \put(60,60){\line(1,0){40}}
                             \put(40,60){\line(1,0){20}}
                             \put(0,60){\line(1,0){20}}
                              \put(-40,60){\circle*{4}}
                              \put(-40,60){\line(1,0){40}}
                              \put(100,60){\line(1,0){40}}
                              \put(140,60){\line(0,-1){30}}
                         \put(140,60){\vector(1,0){40}}
                         \put(137,70){\makebox(0,0){$-1$}}
                         \put(97,70){\makebox(0,0){$-3$}}
                          \put(57,70){\makebox(0,0){$-2$}}
                           \put(-3,70){\makebox(0,0){$-2$}}
                            \put(-43,70){\makebox(0,0){$-2$}}
               %           \put(100,50){\makebox(0,0){$w_{n-1}$}}
              %            \put(-43,50){\makebox(0,0){$w_1$}}
             %             \put(-3,50){\makebox(0,0){$w_2$}}
                          \put(29,57){\makebox(0,0){$\dots$}}
            %              \put(60,50){\makebox(0,0){$w_{n-2}$}}
           %               \put(132,50){\makebox(0,0){$v_2$}}
          %                  \put(132,20){\makebox(0,0){$v_1$}}
                  \put(152,30){\makebox(0,0){$-2$}}
                  \put(-40,60){\vector(-1,0){30}}
         %         \put(-80,70){\makebox(0,0){$a_1$}}
        %          \put(180,70){\makebox(0,0){$a_2$}}
                         \end{picture}

\noindent and the graph of $\partial F$ is

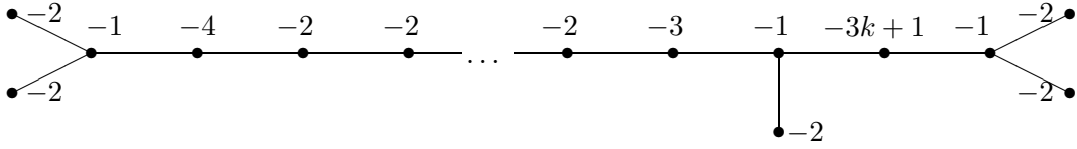
\begin{picture}(300,70)(-160,25)
\put(140,60){\circle*{4}}
                         \put(100,60){\circle*{4}}
                          \put(60,60){\circle*{4}}
                           \put(0,60){\circle*{4}}
                            \put(140,30){\circle*{4}}
                           \put(60,60){\line(1,0){40}}
                             \put(40,60){\line(1,0){20}}
                             \put(0,60){\line(1,0){20}}
                              \put(-40,60){\circle*{4}}
                              \put(-40,60){\line(1,0){40}}
                              \put(100,60){\line(1,0){40}}
                              \put(140,60){\line(0,-1){30}}
                         \put(140,60){\line(1,0){40}}
                         \put(137,70){\makebox(0,0){$-1$}}
                         \put(97,70){\makebox(0,0){$-3$}}
                          \put(57,70){\makebox(0,0){$-2$}}
                           \put(-3,70){\makebox(0,0){$-2$}}
                            \put(-43,70){\makebox(0,0){$-2$}}
                          \put(29,57){\makebox(0,0){$\dots$}}
                  \put(150,30){\makebox(0,0){$-2$}}
                  \put(220,60){\circle*{4}}
                        \put(180,60){\circle*{4}}
                         \put(250,75){\circle*{4}}
                         \put(250,45){\circle*{4}}
                         \put(220,60){\line(2,1){30}}
                         \put(180,60){\line(1,0){40}}
                         \put(220,60){\line(2,-1){30}}
                         \put(213,70){\makebox(0,0){$-1$}}
                          \put(237,75){\makebox(0,0){$-2$}}
                          \put(176,70){\makebox(0,0){$ -3k +1 $}}
                           \put(237,45){\makebox(0,0){$-2$}}
                             \put(-120,60){\circle*{4}}
                                \put(-120,60){\line(-2,1){30}}
                                \put(-120,60){\line(-2,-1){30}}
                                 \put(-80,60){\line(1,0){40}}
                                 \put(-150,75){\circle*{4}}
                                 \put(-150,45){\circle*{4}}
                                  \put(-80,60){\circle*{4}}
                                   \put(-80,60){\line(-1,0){40}}
                                    \put(-138,75){\makebox(0,0){$-2$}}
                                          \put(-138,45){\makebox(0,0){$-2$}}
                                           \put(-80,70){\makebox(0,0){$-4$}}
                                            \put(-115,70){\makebox(0,0){$-1$}}
                         \end{picture}

\noindent
where the number of $(-2)$-vertices in the middle is $n-2$.

 \subsection{$F_{4} $}
 $ \Phi(s, t) = ( s, t^2, s^3t + t^5 )$ and
$ f= y(x^3+y^2)^2 - z^2 =0$.

 
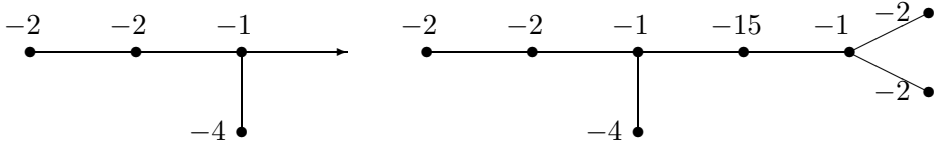
\begin{picture}(300,70)(-200,30)

                           \put(-150,60){\circle*{4}}
                         \put(-110,60){\circle*{4}}
                             \put(-150,60){\line(1,0){40}}
                              \put(-110,60){\line(0,-1){30}}
                              \put(-190,60){\circle*{4}}
                                \put(-110,30){\circle*{4}}
                              \put(-190,60){\line(1,0){40}}
                         \put(-110,60){\vector(1,0){40}}
                           \put(-153,70){\makebox(0,0){$-2$}}
                             \put(-113,70){\makebox(0,0){$-1$}}
          %                   \put(30,50){\makebox(0,0){$w_4$}}
                            \put(-193,70){\makebox(0,0){$-2$}}
            %              \put(-43,50){\makebox(0,0){$w_1$}}
           %               \put(-3,50){\makebox(0,0){$w_2$}}
                           \put(-123,30){\makebox(0,0){$-4$}}
         %                   \put(52,20){\makebox(0,0){$w_3$}}

%\put(-100,60){\makebox(0,0){$\partial F$:}}
                           \put(0,60){\circle*{4}}
                         \put(40,60){\circle*{4}}
                         \put(80,60){\circle*{4}}
                         \put(120,60){\circle*{4}}
                         \put(150,75){\circle*{4}}
                         \put(150,45){\circle*{4}}
                             \put(0,60){\line(1,0){40}}
                              \put(40,60){\line(0,-1){30}}
                              \put(-40,60){\circle*{4}}
                                \put(40,30){\circle*{4}}
                              \put(-40,60){\line(1,0){40}}
                         \put(40,60){\line(1,0){40}}
                          \put(80,60){\line(1,0){40}}
                          \put(120,60){\line(2,1){30}}
                          \put(120,60){\line(2,-1){30}}
                           \put(-3,70){\makebox(0,0){$-2$}}
                             \put(37,70){\makebox(0,0){$-1$}}
                            \put(-43,70){\makebox(0,0){$-2$}}
                             \put(77,70){\makebox(0,0){$-15$}}
                             \put(113,70){\makebox(0,0){$-1$}}
                             \put(136,75){\makebox(0,0){$-2$}}
                              \put(136,45){\makebox(0,0){$-2$}}
                           \put(27,30){\makebox(0,0){$-4$}}
                         \end{picture}

  \subsection{The family $H_{k} $, $k\geq 1$}\labelpar{ss:Hk}
  In this case
$ \Phi(s, t) = ( s, st + t^{3k-1}, t^3 ) $ and the equation of the image can be calculated as the $0^{th}$ fitting ideal of $ \Phi_* ( \mathcal{O}_{ ( \C^2, 0 ) } ) $, see \cite{mondfitting}. One obtains  $ f(x, y, z) = z^{3k-1} - y^3 + x^3 z + 3 xyz^k =0 $.

When $ k>1$, the local form of $ T^2 f $ along $ p ( \tau ) = ( \tau^{3k-2}, - \tau^{3k-1}, \tau^3 ) \in \Sigma^* $ is
\[
 \frac{\tau^{3k-1}}{12 k^2 - 12k + 4}
 \big[-\big((3k-3)- (3k-1) \sqrt{3} i \big) \tau x' +
 \big(3k +(3k-2) \sqrt{3} i\big) y' +
 (6 k^2 - 6k +2 ) \tau^{3k-4} z'
 \big] \cdot \]
\[ \cdot \big[-\big((3k-3)+ (3k-1) \sqrt{3} i\big) \tau x' +
 \big(3k-(3k-2) \sqrt{3} i\big) y' +
 (6 k^2 - 6k +2 ) \tau^{3k-4} z'
 \big]
 \mbox{,}
\]
where $ x'= x-\tau^{3k-2} $, $y'= y+ \tau^{3k-1} $, $ z'= z- \tau^3 $. With the choice
 $ H(x, y, z) = \partial_x f (x, y, z)+\partial_y f(x, y, z)+\partial_z f(x, y, z)$,
 one has $\mathfrak{v} = 3k-1$.

For $ k=1 $ one gets
 \[
 T^2(f) = \big[z' + \big( \frac{3}{2} + \frac{ \sqrt{3}}{2} i \big)
 \tau y' + \sqrt{3} i \tau^2 x' \big] \cdot
 \big[z' + \big( \frac{3}{2} - \frac{ \sqrt{3}}{2} i \big) \tau y' -
 \sqrt{3} i \tau^2 x' \big] \mbox{,}
\]
where $ x'= x-\tau $, $y'= y+ \tau^2 $, $ z'= z- \tau^3 $.
 In this case  $ \mathfrak{v} = 0 $.

In all cases  $ \lambda_1 + \lambda_2 + \mathfrak{v} = -3k-1 $.

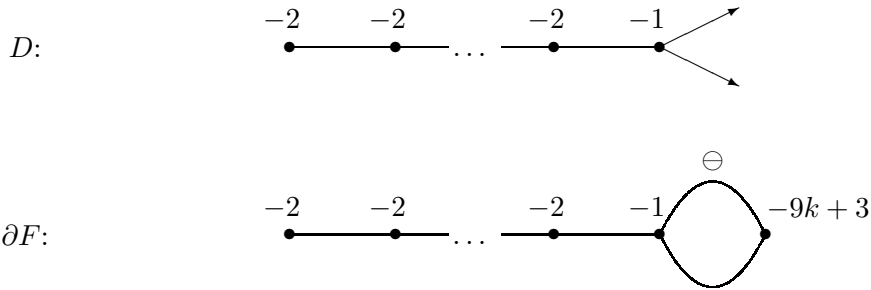
\begin{picture}(300,60)(-150,40)
 \put(-140,60){\makebox(0,0){$D$:}}
                         \put(100,60){\circle*{4}}
                          \put(60,60){\circle*{4}}
                           \put(0,60){\circle*{4}}
                           \put(60,60){\line(1,0){40}}
                             \put(40,60){\line(1,0){20}}
                             \put(0,60){\line(1,0){20}}
                              \put(-40,60){\circle*{4}}
                              \put(-40,60){\line(1,0){40}}
                         \put(100,60){\vector(2,1){30}}
                         \put(100,60){\vector(2,-1){30}}
                         \put(95,70){\makebox(0,0){$-1$}}
                          \put(57,70){\makebox(0,0){$-2$}}
                           \put(-3,70){\makebox(0,0){$-2$}}
                            \put(-43,70){\makebox(0,0){$-2$}}
                          \put(29,57){\makebox(0,0){$\dots$}}
                         \end{picture}

\begin{picture}(300,70)(-150,40)
 \put(-140,60){\makebox(0,0){$\partial F$:}}
                         \put(100,60){\circle*{4}}
                          \put(60,60){\circle*{4}}
                           \put(0,60){\circle*{4}}
                           \put(140,60){\circle*{4}}
                         %  \put(120,60){\circle{40}}
 \qbezier(100,60)(120,100)(140,60)\qbezier(100,60)(120,20)(140,60)
                           \put(60,60){\line(1,0){40}}
                             \put(40,60){\line(1,0){20}}
                             \put(0,60){\line(1,0){20}}
                              \put(-40,60){\circle*{4}}
                              \put(-40,60){\line(1,0){40}}
                         \put(95,70){\makebox(0,0){$-1$}}
                          \put(57,70){\makebox(0,0){$-2$}}
                           \put(-3,70){\makebox(0,0){$-2$}}
                            \put(-43,70){\makebox(0,0){$-2$}}
                             \put(160,70){\makebox(0,0){$-9k +3$}}
                               \put(120,88){\makebox(0,0){$\ominus$}}
                             \put(29,57){\makebox(0,0){$\dots$}}
                         \end{picture}

 \noindent where the number of ($-2$)-vertexes is $ 3k-3$.

In the special case $k=1$ the germs
 $ H_1 $ and $S_1 $ are analytically  equivalent.
 %$ \partial F $:
%
%                         \begin{picture}(300,100)(-100,10)
 %                        \put(100,60){\circle*{4}}
 %                          \put(140,60){\circle*{4}}
 %                          \put(120,60){\circle{40}}
 %                        \put(92,70){\makebox(0,0){$-1$}}
%                             \put(150,70){\makebox(0,0){$-6$}}
%                               \put(120,88){\makebox(0,0){$\ominus$}}
%                         \end{picture}

 \subsection{A corank $2$ map germ}\labelpar{ss:2cor} In this case
$ \Phi (s, t) = (s^2, t^2, s^3 + t^3+ st) $ and $ f(x, y, z) = x^6-2x^4y + x^2y^2 - 2x^3y^3-2xy^4 + y^6 - 8x^2y^2z - 2x^3z^2-2xyz^2-2y^3z^2+z^4 =0 $.

\begin{picture}(300,70)(-50,10)
\put(-40,60){\makebox(0,0){$ D $:}}
                         \put(100,60){\circle*{4}}
                         \put(100,60){\vector(1,0){50}}
                          \put(100,60){\vector(-1,0){50}}
                           \put(100,60){\vector(-1,-1){40}}
                           \put(100,60){\vector(1,-1){40}}
                           \put(100,60){\vector(0,-1){50}}
                         \put(95,70){\makebox(0,0){$-1$}}
                        %  \put(95,50){\makebox(0,0){$v$}}
                         \end{picture}

                         \begin{picture}(300, 250)(-70,-150)
                         \put(-50,100){\makebox(0,0){$\partial F$:}}
                         \put(100,60){\circle*{4}}
                         \put(150,60){\circle*{4}}
                         \put(100,10){\circle*{4}}
                         \put(50,60){\circle*{4}}
                         \put(60,20){\circle*{4}}
                         \put(140,20){\circle*{4}}
                         \put(200,60){\circle*{4}}
                       \put(100,-40){\circle*{4}}
                       \put(0,60){\circle*{4}}
                       \put(20,-20){\circle*{4}}
                       \put(180,-20){\circle*{4}}
                       \put(250,60){\circle*{4}}
                       \put(200,10){\circle*{4}}
                       \put(60,-80){\circle*{4}}
                       \put(140,-80){\circle*{4}}
                       \put(-50,60){\circle*{4}}
                       \put(0,10){\circle*{4}}
                        \put(-30,-20){\circle*{4}}
                         \put(20,-70){\circle*{4}}
                          \put(230,-20){\circle*{4}}
                           \put(180,-70){\circle*{4}}
                         \put(100,60){\line(1,0){50}}
                          \put(100,60){\line(-1,0){50}}
                           \put(100,60){\line(-1,-1){40}}
                           \put(100,60){\line(1,-1){40}}
                           \put(100,60){\line(0,-1){50}}
                          \put(150,60){\line(1,0){50}}
                          \put(100,10){\line(0,-1){50}}
                          \put(50,60){\line(-1,0){50}}
                          \put(60,20){\line(-1,-1){40}}
                          \put(140,20){\line(1,-1){40}}
                              \put(200,60){\line(1,0){50}}
                              \put(200,60){\line(0,-1){50}}
                                  \put(100,-40){\line(-1,-1){40}}
                                   \put(100,-40){\line(1,-1){40}}
                                    \put(0,60){\line(-1,0){50}}
                                    \put(0,60){\line(0,-1){50}}
                                       \put(20, -20){\line(0,-1){50}}
                                          \put(20, -20){\line(-1,0){50}}
                                          \put(180,-20){\line(1,0){50}}
                                          \put(180,-20){\line(0,-1){50}}
                                              \put(95,70){\makebox(0,0){$-1$}}
                                                \put(145,70){\makebox(0,0){$-5$}}
                                                \put(45,70){\makebox(0,0){$-5$}}
                                                 \put(55,30){\makebox(0,0){$-5$}}
                                                  \put(138,30){\makebox(0,0){$-5$}}
\put(107,20){\makebox(0,0){$-5$}}
\put(-5,70){\makebox(0,0){$-1$}}
 \put(195,70){\makebox(0,0){$-1$}}
\put(15,-10){\makebox(0,0){$-1$}}
 \put(179,-10){\makebox(0,0){$-1$}}
 \put(108,-30){\makebox(0,0){$-1$}}
 \put(-55,70){\makebox(0,0){$-2$}}
 \put(-10,20){\makebox(0,0){$-2$}}
\put(-35,-10){\makebox(0,0){$-2$}}
 \put(10,-60){\makebox(0,0){$-2$}}
 \put(55,-70){\makebox(0,0){$-2$}}
\put(138,-70){\makebox(0,0){$-2$}}
\put(245,70){\makebox(0,0){$-2$}}
\put(207,20){\makebox(0,0){$-2$}}
 \put(230,-10){\makebox(0,0){$-2$}}
\put(192,-60){\makebox(0,0){$-2$}}
                         \end{picture}

\end{document}